\documentclass[11pt]{amsart}
\usepackage{amsmath, amsthm}
\usepackage{amssymb}
\usepackage{amscd}
\usepackage{url}
\usepackage{graphicx} 

\numberwithin{equation}{section}

\theoremstyle{plain}

\newtheorem{theorem}[subsection]{Theorem}
\newtheorem{proposition}[subsection]{Proposition}
\newtheorem{lemma}[subsection]{Lemma}

\newcommand{\Q}{\mathbb{Q}}
\newcommand{\Z}{\mathbb Z}
\newcommand{\C}{\mathbb C}
\newcommand{\F}{\mathbb F}

\newcommand{\R}{\mathbb{R}}
\newcommand{\supp}{\operatorname{supp}}

\title{Multifractal behavior of polynomial Fourier series}
\author{Fernando Chamizo and Adri\'{a}n Ubis}
\address{Departamento de Matem\'{a}ticas\\
Facultad de Ciencias
\\
Universidad Aut\'{o}no\-ma de Madrid
\\
28049 Madrid. Spain}
\email{fernando.chamizo@uam.es}
\email{adrian.ubis@uam.es}
\thanks{The authors are partially supported by the grant
MTM2011-22851
from the Ministerio de Ciencia e Innovaci\'{o}n (Spain).}

\begin{document}
\maketitle

\pagestyle{myheadings} 
\markboth{\hfill{\sc F. Chamizo and A. Ubis}}{{\sc Multifractal behavior of polynomial Fourier series}\hfill}

\section{Introduction}

The so called \lq\lq Riemann's example''
\[
 R(x)=\sum_{n=1}^\infty
\frac{\sin(2\pi n^2x)}{n^2}
\]
has a long and fascinating history that has generated a vast literature (see 
\cite{butsta}, \cite[\S1]{chaubi}, \cite{Dui} and references). Just to give a glimpse of it in few
words, we mention that according to Weierstrass \cite{weie}, \cite{edgar}, Riemann
considered $R$ to be an example of a continuous nowhere differentiable function. Indeed
G.H.~Hardy \cite{hardyw} proved in 1916 that $R$ is not differentiable at any irrational value and
at some families of rational values. In 1970, J.~Gerver \cite{gerver} proved, when he was a student,
that $R$ is differentiable at infinitely many rationals. 
The combination of \cite{hardyw} and \cite{gerver} gives a full characterization of the differentiability points of $R$.

More recently several authors have shown interest on the global properties of $R$ and
allied functions (this interest was initially linked to wavelet methods \cite{jafmey},
\cite{holtch}). For instance, it is known that the (box counting) dimension of the graph of $R$ is
$5/4$, in particular it is a fractal \cite{chtams}, \cite{chacor}.

S.~Jaffard \cite{jaffard} proved that $R$ is a \emph{multifractal function}, meaning that if we
classify the points in $[0,1]$ according to the H\"older exponent of $R$,
in the resulting sets we find infinitely many distinct Hausdorff dimensions.
The terminology, introduced firstly in the context of
turbulent
fluid mechanics (see \cite{benzi}),  suggests that a multifractal object is a fractal set with an
intricate structure, containing fractal subsets of different dimensions at different scales.

The multifractal nature of a continuous function $f:[0,1]\longrightarrow\C$ is represented by its
\emph{spectrum of singularities}
\[
 d_f(\beta)=\dim_H\{x\;:\; \beta_f(x)= \beta\}
\]
where $\dim_H$ denotes the Hausdorff dimension and $\beta_f(x)$ is the H\"older exponent of $f$ at
$x$ given by
\[
 \beta_f(x)=\sup\big\{
\gamma\;:\; f\in C^\gamma(x)
\big\}
\]
with 
\[
C^\gamma(x) =\{f\;:\; |f(x+h)-P(h)|
=
O(|h|^\gamma),\text{ for some\;$P\in \C[X]$,\,$\deg P\le \gamma$}
\big\}.
\]
For $\beta_f(x)\le 1$ we have the simpler and usual definition
\[
 \beta_f(x)=\sup\big\{
\gamma\le 1\;:\; |f(x+h)-f(x)|=O(|h|^\gamma)
\big\}.
\]
We let $d_f(\beta)$ undefined if $\{x\;:\; \beta_f(x)=\beta\}$ is the empty set.
In this way, the domain of
$d_f$ is always a subset of $[0,\infty )$.

For a \lq\lq purely fractal'' function as the celebrated Weierstrass nondifferentiable function,
the graph of $d_f$ consists of a finite number of points while for a multifractal function we observe a non-discrete
graph \cite{jaffard2}.

\

The main results in \cite{jaffard} are summarized saying that for a given $\alpha>1$ the
following function (already appearing in early works of Hardy and Littlewood \cite{hardlitt})
\[
 R_\alpha(x)=\sum_{n=1}^\infty\frac{\sin(2\pi n^2x)}{n^\alpha}
\]
has the following spectrum of singularities revealing its multifractal nature.

\begin{center}
\includegraphics[scale=0.41]{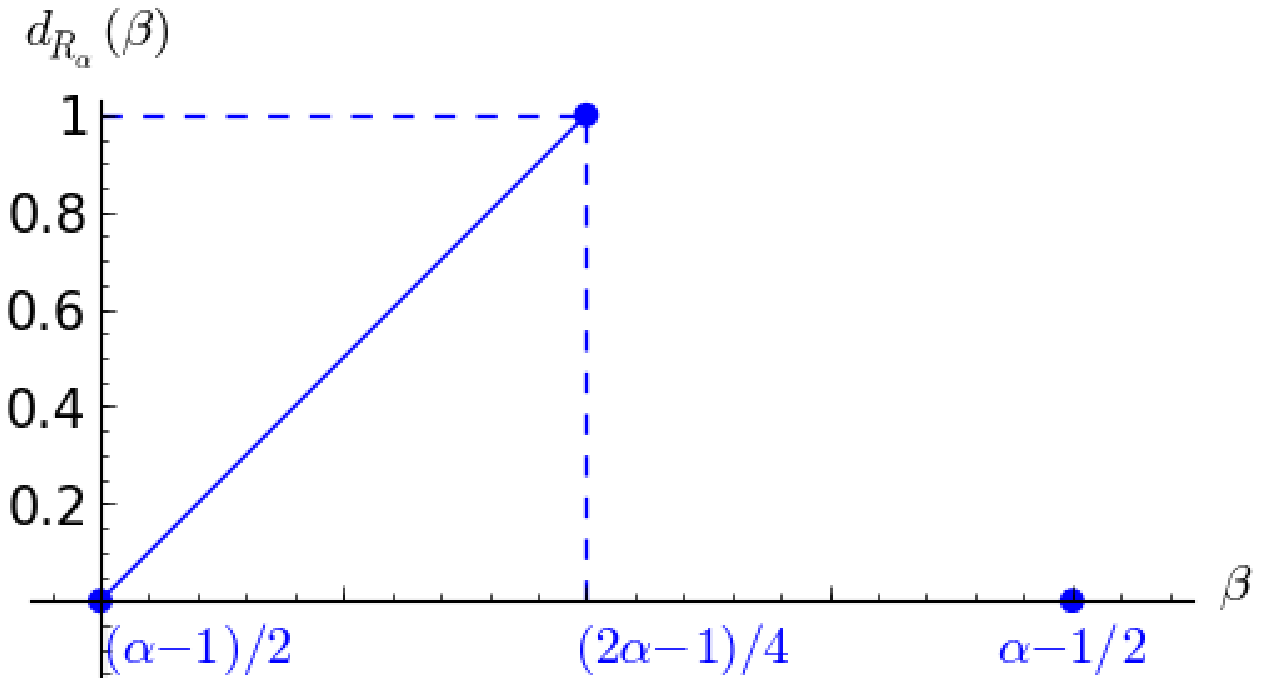}
\end{center}

For instance, taking $\alpha=2$ we deduce  that there are sets of points with increasing Hausdorff
dimension and H\"older exponents ranging from $1/2$ to $3/4$. Moreover there is a $0$-dimensional
set in which the H\"older exponent is $\alpha-1/2=3/2$, in particular $R$ is differentiable in it.
In fact, \cite{gerver} proves that this set contains a certain explicit subset $S\subset\Q$ and
\cite{hardyw} that the H\"older exponent is not greater than $3/4$ in $\R-S$. In this way, these
results prove that for $\beta>3/4$ the graph of $d_R(\beta)$ is a simple point
at $(3/2,0)$.

\

In this paper we study the spectrum of singularities of
\begin{equation}\label{deff}
 F(x)=\sum_{n=1}^\infty\frac{e\big(P(n)x\big)}{n^{\alpha}}
\end{equation}
where $e(t):=e^{2\pi i t}$ and $P\in \Z[X]$ is a polynomial of degree $k>1$. 
Note that for $k=1$ the formula defines a $C^\infty$ function outside a discrete set. The functions $R$,
$R_\alpha$ and $F$ that we have considered belong to the rather unknown realm of the Fourier series
with gaps whose coefficients decay too slow to be $C^\infty$ and whose frequencies do not grow quick
enough to be lacunary series in the classical sense \cite[Ch.V]{zygm}.
In harmonic analysis there are old conjectures (by W.~Rudin \cite{rudin}) suggesting that gaps in the
frequencies induce some regularity even for general coefficients.

For $k=2$, $R_\alpha$ is no other than the imaginary part of $F$. 
The differences between $R_\alpha$ and $F$ for $k>2$ are fundamental. Basically $R_\alpha$ is a
fractional integral of the automorphic theta function $\theta(z)=\sum_{n=-\infty}^\infty e(n^2z)$
and the local behavior of $R_\alpha$ at a given $x$ is determined by the convergents in the
continued fraction of $x$ and by the Fourier expansion of $\theta$ at the cusps (every rational
number is equivalent to a cusp). This approach is developed in \cite{chtams} and \cite{milsch} in a
broader context.

For $k>2$ there is no underlying automorphic function. Actually one can control locally $F$ only
in the intervals $(a/q-h,a/q+h)$ with $h<q^{-k}$ whose union is  a set of Hausdorff dimension $2/k$
(see \cite[\S1]{chaubi} and Jarn\'{\i}k's
theorem \cite{Fal}) and it has positive measure only for $k=2$. A more important barrier to treat
the case $k>2$ is that using Poisson summation, one improves the trivial estimate only when 
$h<q^{-k/2}$ (see the comments in the introduction of \S3) while for rational approximations
of a point we have to deal with $h$ almost like $q^{-1}$ (cf. Lemma~\ref{contfrac} with large $r$).

To overcome these difficulties one has to go beyond the local analysis at individual points,
considering instead the problem globally, in average. But it is important to keep in mind that the
spectrum of singularities requires to deal with subsets having fractional  Hausdorff dimension. Then this average has to be done in restricted sets and, for instance, integration that is useful to
compute the dimension of the graph of $F$ \cite{chacor} \cite{chaubi} is too coarse here. Once this
fine average is carried out, we can construct some fractal sets whose elements have
special  diophantine approximation properties  that allow to extract a main term for the variation of $F$ in some ranges. 
Such a main term depends in some way on sums of the form
$q^{-1/2}\sum_{n=1}^qe\big(aP(n)/q\big)$. The purely arithmetic  fact that this sum is typically
greater for $q=p^n$ than for $q=p$ prime if $P$ has multiple zeros modulo $q$,  suggests  an
unexpected dependence on the maximal multiplicity $\nu_F$ of
a  (complex) zero of $P'$. Namely on
\[
 \nu_0=\max(\nu_F,2).
\]

\medskip

Our main result is a lower bound for the spectrum of singularities of $F$.  It is worth remarking that in this context any lower bound is highly non-trivial. The dependence on
$\nu_0$ allows to give the same result for $P(n)=n^2$ (Riemann
example) and $P(n)=n^3$. This is noticeable taking into account that we expect the bound to be
sharp in some ranges.

\begin{theorem}\label{mainth}
Consider $F$ as before with $1+k/2<\alpha<k$. Then for $0\le \beta< 1/2k$
\[ 
d_F\big(\beta+\frac{\alpha-1}{k}\big)\ge (\nu_0+2)\beta.
\]
\end{theorem}

\

We complement this result with a upper bound that in particular implies that $F$ is actually a
multifractal function. 

\begin{theorem}\label{mainth_upper}
Let $I= \big[0, 1/2k\big]$ with the ranges as before. 
There exists $\omega:I\longrightarrow [0,1]$ continuous at $0$  and strictly increasing
with $\omega(0)=0$ and $w(\frac{1}{2k}^-)=1$ such that $d_F\big(\beta+\frac{\alpha-1}{k}\big)\le
\omega(\beta)$. In fact
\[
\omega(\beta)=
\begin{cases}
\frac{2\beta}{2^{-k}+\beta}&\text{ if } 0\le \beta< \frac{1}{k2^{k}}
\\
\frac 32-\sqrt{\frac{k+4}{4k}-2\beta} &\text{ if } \frac{1}{k2^{k}}\le\beta<\frac{1}{2k}
\end{cases}
\]
is a valid choice.
\end{theorem}

Note that, for $\beta\in
\big[0, k^{-1}2^{-k}\big)$, Theorem~\ref{mainth} and Theorem~\ref{mainth_upper} imply that 
$d_F\big(\beta+\frac{\alpha-1}{k}\big)$  is bounded between two continuous functions that vanish at $\beta=0$  and
hence $d_F$ cannot take a finite discrete set of values.

\medskip

Geometrically, we have that the graph of $d_F$ is contained in the
shadowed region in the indicated range. We
think that for $\beta$ small, the inequality in Theorem~\ref{mainth} is actually an equality (see \S7).

\begin{center}
\begin{tabular}{c}
\includegraphics[scale=0.41]{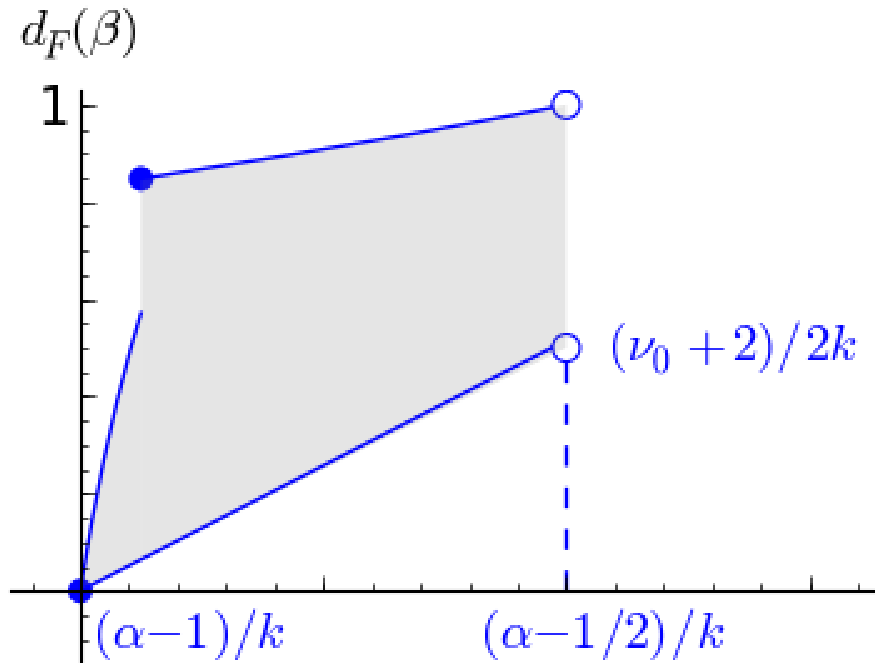}
\end{tabular}
\end{center}

We have restricted $\alpha$ in our main theorem to the range $1+k/2<\alpha<k$. The  bound $1+k/2$ is
intrinsic to the method. The upper bound~$k$ is not critical but it would require first to apply a wavelet transform to $F$ in order to separate 
whole derivatives,
as in \cite[Corollary~1.33]{phdubis} for $P(n)=n^k$, and then to use our analysis here.

\

The structure of the paper is as follows. In \S2 we study locally $F$ by Poisson summation formula
and exponential sums methods, getting as a by-product the exact spectrum of singularities for
$R_\alpha$
(this gives a simplified
proof of some results in \cite{jaffard}). In \S3 we state several results to control
the oscillation of $F$ for most rationals in some restricted sets, beyond the ranges obtained in
\S2. The idea to get the lower bound for $d_F$ is to
perform an approximation process using special rational values to construct a subset $A\subset
[0,1]$ such that $\beta_F(x)$ is fixed for every $x\in A$. The process leads to a Cantor-like set
construction and we devote \S5 to define generalized Cantor sets and to compute their Hausdorff
dimensions. We need special arithmetic considerations about exponential sums and
the spacing of the elements of some sets,  that are included in \S4. Finally in \S6 we combine all
the
tools to prove Theorem~\ref{mainth} and Theorem~\ref{mainth_upper}. 
As an appendix, we include a last section with some heuristics and conjectural properties of the spectrum of singularities of $F$.
In fact, the reader can find convenient to read it after this introduction to learn our motivation and the limitations of our method.

\

The approximation process transfers our knowledge about the oscillation at rational values to the
selected points. The behavior of some exponential sums imposes some restrictions on the
denominators, but we think that the method is lossless in terms of the dimension  (see \S7). 

\

The idea of using averages over rationals in this problem was first developed in \cite[Chapter 1]{phdubis}, 
for the case $P(x)=x^k$. We would like to take this opportunity to point out an important error in Theorems 1.3 and~1.4 (restated in \S1.6) there. It  is claimed that $\beta_F(x)=w$ for any $x$ in a set of positive Hausdorff dimension and what can be really demonstrated is just $\beta_F(x)\le w$
because for the lower bounds of the Hausdorff dimension one would need restricted averages as in Proposition~\ref{osc3} here.

The study of the spectrum of singularities of $F$ for $P(x)=x^k$ via Poisson summation was initiated in  \cite[\S4]{chaubi}. We would like to remark that there is a small typo in that paper: we defined $s(x)=\liminf_n s_n$ while it should be defined in the same way but restricting $n$ to subsequences $n_k$ for which $\lim_k r_{n_k} =r(x)$.

\

\section{Local analysis}

In this section we tackle two problems: The approximation of $F$ around rational values and upper
bounds for $\beta_F(x)$ in terms of the approximation of $x$ by rationals. In both cases we assume
that the
leading coefficient of $P$ is $c_0>0$. Indeed this is not an actual restriction because the
sign changes under conjugation.

\smallskip

We address the first problem via Poisson summation formula. Our first result is a simple general statement adapted to our setting and the second is a consequence after the estimation of some oscillatory sums and integrals

\begin{proposition}\label{poisson}
Assume $1\le k/2<\alpha$. Then for any 
$0<h<1$ and any irreducible fraction $a/q$, $0\le a<q\le1$, we have the absolutely convergent
expansion
\[
F\big(\frac aq+h\big)-F\big(\frac aq\big)=
q^{-1}h^{(\alpha-1)/k}
\sum_{m=-\infty}^\infty
\tau_m\widehat{g}_h\big(h^{-1/k}q^{-1}m\big)
\]
where
\[
 \tau_m =\sum_{n=1}^q
e\big(\frac{aP(n)+mn}{q}\big)
\quad\text{and}\quad
g_h(x)=\Phi\big(h^{-1/k}x\big)
\frac{e\big(hP(h^{-1/k}x)\big)-1}{x^\alpha}
\]
with $\Phi\in C^\infty$, $\supp\Phi\subset\R^+$, 
$\Phi\big|_{[1,\infty)}=1$
and 
$\widehat{g}_h(\xi)=\int g_h(x)e(-\xi x)\; dx
=O\big((1+|\xi|)^{-\delta}\big)$ for some $\delta >1$.
\end{proposition}

Note that for $x$ positive, $\lim_{h\to 0^+}g_h(x)=x^{-\alpha}\big(e(c_0x^k)-1\big)$. The condition
$k/2<\alpha$ is only to assure easily the convergence but can be relaxed (see the proof). It is
important to note for future applications of this result that $h>0$ is not an actual restriction
because $F(a/q-h)-F(a/q)$ is the complex conjugate of $F\big((q-a)/q+h\big)-F\big((q-a)/q\big)$. 

In subsequent applications the   main term will come from $m=0$.

\begin{proposition}\label{asymp}
Under the hypotheses of the previous result, assuming also $\alpha<k$ and $q$ prime,  we have
\[
F\big(\frac aq+h\big)-F\big(\frac aq\big)
=A\frac{\tau_0}{q}(c_0h)^{(\alpha-1)/k}+O\big(h^{\alpha/k}q^{1/2}\big)
\]
where the $O$-constant only depends on~$F$ and 
\[
A=\frac{(2\pi)^{(\alpha-1)/k}}{k}e\big(\frac{1-\alpha}{4k}\big)\Gamma\big(\frac{1-\alpha}{k}
\big).
\]
Moreover the result extends to $\alpha=k$ introducing a factor $|\log h|$ in the error term.
In this extended range $k/2<\alpha\le k$, the result still applies for~$q$ square-free introducing an extra $q^\epsilon$ factor in the error term ($\epsilon>0$) and allowing the $O$-constant to depend on $\epsilon$. In the special case $P(x)=c_0x^k$ or for any polynomial of degree 2, this latter form of the result also holds for any $q\ge 1$.
Indeed, the quadratic case holds with $\epsilon=0$.
\end{proposition}

As we mention in the introduction, this form of Poisson summation formula suffices to give
a short and simple proof of one of the main results in \cite{jaffard} (Corollary~2), with a slightly
more general function but a more restrictive range.

\begin{theorem}\label{jaffard}
If $P$ is a polynomial of degree $2$ then
\[ 
d_F\big(\beta+\frac{\alpha-1}{2}\big)=4\beta
\qquad\text{ for }\quad 0\le \beta\le \frac 14 \quad\text{ and }\quad 1<\alpha\le 2.
\]
On the other hand, $\beta_F(x)\le (\alpha-1)/2+1/4$ for any irrational $x$, in particular $d_F\big(\beta+\frac{\alpha-1}{2}\big)$ is zero or remains undefined for $\beta>1/4$.
\end{theorem}

\medskip

We treat the second problem using Weyl's inequality \cite{Vau} to obtain the following bound. The
constant $c_0$ could be omitted in the statement and the bound still holds true but it is natural
for a direct application of  Weyl's inequality.

\begin{proposition}\label{weyl}
Given $r>2$ and  $x_0\not\in\Q$  such that the number of irreducible fractions $a/q$
satisfying $|c_0x_0-a/q|<q^{-r}$ is finite. Then for $1<\alpha\le k+1/2$
\[
\beta_F(x_0)\ge
\frac{\alpha-1}{k}
+2^{1-k}
\min\big(\frac{1}{k},\frac{1}{2(r-1)}\big).
\]
\end{proposition}

\

We finish by stating an analogue of Proposition~\ref{asymp} in the case $\nu_F>1$ that works in a larger range of $h$ but just for a part of $F$.

\begin{proposition}\label{pointw}
Given $p>k>\alpha$, 
let $\nu_{F,p}$ the maximal multiplicity of a zero of $P'$ modulo $p$.
Assume
$\nu_{F,p}=\nu_F>1$. Then, there exists some $\gamma>0$ such that for all $1\le a\le q=p^{\nu_F+1}$
but at most $q^{1-\gamma}$ exceptions we have
\[
 F^*(\frac aq+h)-F^*(\frac aq)=\frac{\tau_*}p A {(c_0h)}^{(\alpha-1)/k}
+{ O\big(p^{-1}h^{(1+\gamma)(\alpha-1)/k}\big)}
\]
for any $h<p^{-2k}$, where $F^*$ is the sum in \eqref{deff} restricted to $\big\{n\;:\; p\mid P'(n)\big\}$ and
\[
\tau_*=\sum_{b\in \mathcal{B}} e(\frac{a P(b)}q)
\]
with $\mathcal{B}$ the set of zeros of multiplicity $\nu_F$ of $P'$ modulo $p$.
\end{proposition}

In fact the exceptions $a/q$ in the previous result are essentially described in its proof.

\medskip

For the proof of Proposition~\ref{asymp} we state separately a well-known smoothing device and the
asymptotic expansion of a certain integral related to the main term.

\begin{lemma}[Dyadic smooth subdivision]\label{d_s_s}
There exists $\phi\in C_0^\infty(\R)$ with $\supp\phi\subset [1/2,2]$ such that
$\Phi(x)=\sum_{j=0}^\infty\phi(x/2^j)$ verifies $\Phi(x)=1$ for $x\ge 1$. 
\end{lemma}

\begin{proof}
Take $f(x)=\int_{1/2}^x e^{-(2u-1)^{-2}(u-1)^{-2}}\; du$ and choose for instance,
$\phi(x)=f(x)/f(1)$ in $[1/2,1]$, $\phi(x)=1-f(x/2)/f(1)$ in $[1,2]$ and $\phi(x)=0$ in the rest of
the points. 
Note that $\phi\in C_0^\infty(\R)$ and satisfies  $\phi(x)=1-\phi(x/2)$.
\end{proof}

\begin{lemma}\label{integral}
Let $P\in\R[X]$ be monic of degree $k$ and $1<\alpha<k$. For $0<h<1$, we have
\[
 \int_1^{\infty} \frac{e\big(hP(u)\big)-1}{u^{\alpha}}\; du
=Ah^{(\alpha-1)/k}+O\big(h^{\alpha/k}\big),
\]
with $A$ the constant in the statement of Proposition~\ref{asymp},
and the result extends to $\alpha=k$ substituting the error term by $O\big(h|\log h|\big)$.
\end{lemma}

\begin{proof}
The identity $\Gamma(z)=\int_0^\infty
t^{z-1}\big(e^{-t}-1\big)\; dt$, valid for $-1<\Re(z)<0$, follows from the integral representation
of the $\Gamma$-function by partial integration. 
Applying residue theorem,
\[
 Ah^{(\alpha-1)/k}=
\frac{h^{(\alpha-1)/k}}{k}
\int_0^\infty t^{(1-\alpha-k)/k}\big(e(t)-1)\;dt
=
\int_0^\infty u^{-\alpha}\big(e(hu^k)-1)\;du.
\]
For $0<u<1$ the last integral contributes $O(h)$. Hence it is enough to prove
\[
I= O\big(h^{\alpha/k}\big)
\quad\text{for }
 I= 
\int_1^{\infty} 
u^{-\alpha}e\big(hu^k\big)\Big(e\big(hQ(u)\big)-1\Big)\; du
\]
where $Q(u)=P(u)-u^k$. Note that $\deg Q<k$.

Let $\phi\in C_0^\infty(\R)$ with $\supp \phi=[1,2]$ and  $\int \phi=1$ and
let $\Phi_+(x)=\int_0^x\phi$ and $\Phi_-(x)=1-\Phi_+(x)$. 
We write $I=I_{-}+I_{+}$ where
\[
 I_{\pm}= 
\int_1^{\infty} 
\Phi_{\pm}(h^{1/k}u)u^{-\alpha}e\big(hu^k\big)\Big(e\big(hQ(u)\big)-1\Big)\; du.
\]
The bound $\big|e\big(hQ(u)\big)-1\big|\ll hu^{k-1}$ gives $I_{-}=O\big(h^{\alpha/k}\big)$ if
$\alpha<k$ and $I_{-}=O(h|\log h|)$ if $\alpha=k$.

On the other hand, 
\begin{eqnarray*}
I_+&=&
2\pi i \int_0^h\int 
\Phi_+(h^{1/k}u)u^{-\alpha}
Q(u)
e\big(hu^k+\xi Q(u)\big)\; du\;d\xi
\\
&=&
2\pi i h^{-1+\alpha/k}
\int_0^h\int 
\Phi_+(u)u^{-\alpha}
\frac{Q(h^{-1/k}u)}{h^{(1-k)/k}}
e\big(u^k+\xi Q(h^{-1/k} u)\big)\; du\;d\xi.
\end{eqnarray*}
The inner integral is $O(1)$ integrating by parts, because the derivative of the phase is bounded
from below \cite{grahkole}, and we get
$I_{+}=O\big(h^{\alpha/k}\big)$.
\end{proof}

\begin{proof}[Proof of Proposition~\ref{poisson}]
Note that we can write 
\[
 F\big(\frac aq+h\big)-F\big(\frac aq\big)
=\sum_{j=1}^q
e\big(\frac{aP(j)}{q}\big)
\sum_{n\equiv j\ (q)}\Phi(n)\frac{e\big(hP(n)\big)-1}{n^\alpha}. 
\]
Poisson summation formula for arithmetic progressions \cite[\S4.3]{IwKo} implies, assuming the
absolute convergence
\[
 F\big(\frac aq+h\big)-F\big(\frac aq\big)
=q^{-1}
\sum_{m=-\infty}^\infty
\tau_m\widehat{f}(m/q)
\]
where
\[
 \widehat{f}(\xi)= \int \Phi(x)\frac{e\big(hP(x)\big)-1}{x^\alpha}e(-\xi x)\; dx.
\]
Changing variables $x\mapsto h^{-1/k} x$, we get the expected formula with
$\widehat{g}_h(\xi)=h^{(1-\alpha)/k}\widehat{f}\big(h^{1/k}\xi\big)$.

\smallskip

Finally note that for large $\xi$ the integral defining $\widehat{g}_h$ has a single stationary
point at $x\sim (k^{-1}\xi)^{1/(k-1)}$ if $\xi>0$. Then the principle of stationary phase
\cite{sogge}, \cite{erdelyi} or simpler considerations \cite[\S3.2]{grahkole} prove
$\widehat{g}_h(\xi)\ll_h\xi^{(1-\alpha-k/2)/(k-1)}$
and a better bound for $\xi<0$. Under our assumption $\alpha>k/2$, we have
$\widehat{g}_h(\xi)=O\big(|\xi|^{-1- \epsilon}\big)$ that assures the absolute convergence after the
trivial
bound $|\tau_m|\le q$.
\end{proof}

\begin{proof}[Proof of Proposition~\ref{asymp}]
Take $\Phi$ as in Lemma~\ref{d_s_s}, then we can write 
Proposition~\ref{poisson} as
\[
 F\big(\frac aq+h\big)-F\big(\frac aq\big)
=q^{-1}\sum_{N=2^j}^\infty
\sum_{m=-\infty}^\infty
\tau_m\widehat{f}_N(m/q)
\]
where
\[
 f_N(x)= \phi\big(\frac{x}{N}\big)\frac{e\big(hP(x)\big)-1}{x^\alpha}.
\]

The contribution of $m=0$ in the sum is
\[
 \frac{\tau_0}{q}
 \int_0^{\infty}\Phi(x) \frac{e\big(hP(x)\big)-1}{x^{\alpha}}\; dx
= \frac{\tau_0}{q}
\int_1^{\infty} \frac{e\big(hP(x)\big)-1}{x^{\alpha}}\; dx+O(h)
\]
that gives the main term by Lemma~\ref{integral}.

Weil's bound \cite{schmidt} proves $\tau_m\ll q^{1/2}$. Hence after
a change of variables
\begin{equation}\label{ps_f}
 F\big(\frac aq+h\big)-F\big(\frac aq\big)-A\frac{\tau_0}{q}(c_0h)^{(\alpha-1)/k}
\ll q^{-1/2}\sum_{N=2^j}^\infty
N^{1-\alpha}\sum_{m\ne 0}
|g_N(m/q)|
\end{equation}
with 
\[
g_N(\xi)=\int x^{-\alpha}\phi(x)e(-\xi Nx)\big(e(hP(Nx))-1)\; dx. 
\]

If $N^kh\le 1$, we proceed as in the proof of Lemma~\ref{integral} writing
\[
g_N(\xi)=
2\pi i N^{k}\int_0^h\int  
x^{-\alpha}\phi(x)\frac{P(Nx)}{N^k}e\big(-\xi Nx+uP(Nx)\big)\; dx \; du.
\]
When $N|\xi|<1$ the trivial estimate $O\big(N^{k})$ for the inner integral is sharp.
Otherwise, we save $N^2|\xi|^2$ integrating by parts twice. Then
\begin{eqnarray}\label{ps_f_1}
\sum_{m\ne 0}
|g_N(m/q)|
&\ll& N^{k}h\Big(\sum_{m<q/N}1+\sum_{m\ge q/N}N^{-2}q^2m^{-2}\Big)
\\
&\ll & hqN^{k-1}\qquad\text{for }\quad N^kh\le 1.\notag
\end{eqnarray}
Note that for $\alpha=k$ this contributes $O\big(q^{1/2}h|\log h|\big)$ in the right hand side of
\eqref{ps_f}.

If $N^kh\ge 1$, we separate $g_N$ in two integrals
\[
g_N(\xi)=
\int x^{-\alpha}\phi(x)e(hP(Nx)-\xi Nx)\; dx
-
\int x^{-\alpha}\phi(x)e(-\xi Nx)\; dx
\]
In the range $C^{-1}hN^{k-1}\le |\xi|\le ChN^{k-1}$ with $C>0$ a large constant, stationary phase
method \cite{sogge} (or van der Corput estimate) proves that the first integral is
$O\big(h^{-1/2}N^{-k/2}\big)$. If $|\xi|<C^{-1}hN^{k-1}$, it is $O\big(h^{-1}N^{-k}\big)$
integrating by parts once. Finally for $|\xi|> ChN^{k-1}$ 
in the first integral, and in the whole range in the second integral, 
we proceed as before 
using the trivial estimate and double integration by parts to get
$O\big(\min(1,h^{2/k}|\xi|^{-2})\big)$. 

These bounds give
\begin{eqnarray}\label{ps_f_2}
\sum_{m\ne 0}
|g_N(m/q)|
&\ll& 
\sum_{m\ll qhN^{k-1}}h^{-1}N^{-k}
+
\sum_{m\asymp qhN^{k-1}}h^{-1/2}N^{-k/2}
\\
&\ & 
+\sum_{m\gg qhN^{k-1}}\min(1,h^{2/k}q^2m^{-2}) \notag
\\
&\ll & 
q\big(h^{1/2}N^{k/2-1}+h^{1/k}\big)\qquad\text{for }\quad N^kh\ge 1.\notag
\end{eqnarray}

Substituting \eqref{ps_f_1} and \eqref{ps_f_2} in \eqref{ps_f} we obtain the expected result.

\

If $q$ is square-free then the multiplicative properties of $\tau_m$ (e.g. Th.1 \cite{smith}) allow to  
replace $q$ by each of its prime factors and apply Weil's bound,
resulting $|\tau_m|\le K^{d(k)}q^{1/2}\ll q^{1/2+\epsilon}$ where $d(q)$ is the number of divisors of $q$.

In the case $P(x)=c_0x^k$, we have Hua's bound \cite{loxsmi} that reads $\tau_m\ll d_{k-1}(q)q^{1/2}(m,q)^{1/2}$ with $d_{k-1}$ the $k-1$ divisor function \cite{ivic}. The elementary bound $\sum_{m\asymp M}(m,q)\ll d(q)M$ shows that \eqref{ps_f_1} and \eqref{ps_f_2} are still true with $d(q)$ in the right hand side when introducing $(m,q)^{1/2}$ in the summation.

For the quadratic case, the $\epsilon$-free bound $\tau_m\ll q^{1/2}$ for Gaussian sums is elementary.
\end{proof}

\medskip

\begin{proof}[Proof of Theorem~\ref{jaffard}]
Let $\{a_n/q_n\}_{n=1}^\infty$ be the convergents of $x\not\in\Q$. We consider the sets
\[
A_r=\big\{ 
x\in [0,1]\setminus\Q\;:\;
\big|x-\frac{a_n}{q_n}\big|=\frac{1}{q_n^{r_n}}
\quad\text{ with }\limsup r_n=r
\big\}
\]
and $A_r^*$ defined in the same way but adding the condition $2\nmid q_{n_k}$ for some subsequence
with $\lim r_{n_k}=r$. 
Note that $|x-a_n/q_n|<q_n^{-2}$ implies $r_n>2$.
By simple variants of 
Jarn\'{\i}k's theorem \cite{Fal},
$
\dim_H A_r=
\dim_H A_r^*=
\dim_H \bigcup_{s\ge r}A_s=2/r
$
for $2\le r\le \infty$.

Given $x\in A_r$ and any small $h\ne 0$, there exists $n$ such that
\[
\frac{1}{q_n^{r_n}}
=
\big|x-\frac{a_n}{q_n}\big|
\le |h|<
\big|x-\frac{a_{n-1}}{q_{n-1}}\big|
=
\frac{1}{q_{n-1}^{r_{n-1}}}.
\]
Hence, as $(2q_nq_{n-1})^{-1}>q_n^{-r_{n-1}}$, we have
$|h|^{-1/r_n}\le q_n< |h|^{-1+1/r_{n-1}}$.

The evaluation of the Gauss sums (cf. \cite[\S3]{IwKo}) or more elementary arguments, prove that  $\tau_0\gg \sqrt{q}$ for $q$ odd. Then by Proposition~\ref{asymp} and the bounds on $q_n$, we have, assuming $\alpha>2$ and writing $s_n=\max(r_{n-1},r_n)$,
\begin{equation}\label{jaf_u}
F(x+h)-F(x)
=
F(x+h)-F\big(\frac{a_n}{q_n}\big)
-
\big(F(x)-F\big(\frac{a_n}{q_n}\big)\big)
\ll h^{(\alpha-1)/2+1/2s_n} 
\end{equation}
and for the special choice $h=h_n=x-a_n/q_n$, if $q_n$ is odd and $q_n^{r_n-2}$ is greater than a large enough fixed constant,
\begin{equation}\label{jaf_l}
F(x)-F(x-h_n)
=
F\big(\frac{a_n}{q_n}+h_n\big)
-
F\big(\frac{a_n}{q_n}\big)
\gg h^{(\alpha-1)/2+1/2r_{n}}. 
\end{equation}

Taking $x\in A_r^*$ and $n=n_k$ in \eqref{jaf_u} and \eqref{jaf_l} we deduce that  $\beta_F(x)$ is
exactly $(\alpha-1)/2+1/2r$ and
\[ 
d_F\big(\frac{\alpha-1}{2}+\frac{1}{2r}\big)\ge \dim_H A_r^*=\frac{2}{r}
\qquad
\text{for }\quad 2<r\le \infty.
\]
From \eqref{jaf_u} we have that any $x$ with H\"older exponent $(\alpha-1)/2+1/2r$ is contained in $\Q\cup\bigcup_{s\ge r}A_r$. Since this is a set of dimension $2/r$ we have 
\[ 
d_F\big(\frac{\alpha-1}{2}+\frac{1}{2r}\big)\le \frac{2}{r}
\]
This proves the result for $\alpha<2$ taking $\beta=1/2r$ except in the case $\beta =1/4$. 

\smallskip

Any
$x\in[0,1]\setminus\Q$ is in some $A_s$, $2\le s\le \infty$.
Take $h_n=\epsilon|x-a_n/q_n|$ with $\epsilon>0$ a small enough constant, then the last inequality in \eqref{jaf_l} still holds for $q_n$ odd even if 
$q_n^{r_n-2}$ is not large and we have 
\begin{eqnarray*}
\Big( F\big(\frac{a_n}{q_n}+h_n\big)
-
F\big(x\big)\Big) 
+\Big( F\big(x\big)
-
F\big(\frac{a_n}{q_n}\big)\Big) 
&=&
F\big(\frac{a_n}{q_n}+h_n\big)
-
F\big(\frac{a_n}{q_n}\big)
\\
&\gg& h^{(\alpha-1)/2+1/4}
\end{eqnarray*}
that proves $\beta_F(x)\le  (\alpha-1)/2+1/4$ for $x\in[0,1]\setminus\Q$ unconditionally because $2\mid q_n$ implies $2\nmid q_{n+1}$.

For the case $\beta=1/4$ it only remains to note that \eqref{jaf_u} implies $\beta_F(x)\ge  (\alpha-1)/2+1/4$ in $A_2$ which is a set of full dimension.

Finally, for $\alpha=2$ the same argument applies when $\beta<1/4$ introducing a harmless logarithmic factor in \eqref{jaf_u}, and the case $\beta =1/4$ is treated in the same way but choosing this time $h_n(\log h_n)^2=\epsilon|x-a_n/q_n|$.
\end{proof}

\begin{proof}[Proof of Proposition~\ref{weyl}]
Applying the mean value theorem to the real and imaginary parts of the series defining $F$ we have
\begin{equation}
\label{decompF}
F(x_0+\frac h{c_0})-F(x_0)
\ll 
h
\big|
\sum_{n\le h^{-1/k}} n^{k-\alpha} e\big(P(n)\xi\big)
\big|
+
\big|
\sum_{n> h^{-1/k}} n^{-\alpha} e\big(P(n)\xi\big)
\big|
\end{equation}
for some $\xi\in [x_0,x_0+h/c_0]$ where we have assumed $h>0$ (the other case is similar).

\smallskip

Recall 
that Weyl's inequality assures that if $Q$ is a polynomial of degree $k$  with leading coefficient
$\lambda$ such that $|\lambda-a/q|\le q^{-2}$ for an irreducible fraction $a/q$, then 
\cite[Lemma~2.4]{Vau}
\begin{equation}\label{weyl_ineq}
\sum_{n\le N}  e\big(Q(n)\big)
\ll_{Q,\epsilon}
\big(
Nq^{-2^{1-k}}+N^{1-2^{1-k}}+N^{1-k2^{1-k}}q^{2^{1-k}}
\big)N^{\epsilon}
\text{ for every }\epsilon>0.
\end{equation}

\smallskip

Consider two consecutive convergents of $c_0x_0$ such that 
\begin{equation}\label{consec}
\big|c_0x_0-\frac{a_{n}}{q_{n}}\big|
<h\le
\big|c_0x_0-\frac{a_{n-1}}{q_{n-1}}\big|;
\end{equation}
with $n$ sufficiently large.

By our assumption on $x_0$ and elementary properties of the continued fractions, we have
\begin{equation}\label{propert}
 q_n^{-r}\le h<q_{n-1}^{-2}
\qquad\text{and}\qquad
q_{n-1}^{-r}<q_n^{-1}q_{n-1}^{-1}.
\end{equation}

If $h\le \frac 12 q_n^{-2}$ then $|c_0\xi-a_n/q_n|<h<q_n^{-2}$ and we can take $q=q_n$ in Weyl's
inequality \eqref{weyl_ineq}. 
Noting that $h^{-1/r}\le q_n\le (2h)^{-1/2}$, we deduce  $\beta_f(x_0)\ge
(\alpha-1)/k+2^{1-k}\min(1/r,1/k)$ from \eqref{decompF} by partial summation.

If $h> \frac 12 q_n^{-2}$ then $q_{n-1}>q_{n}^{1/(r-1)}\gg h^{-1/2(r-1)}$. The inequalities
\[
\big|c_0\xi-\frac{a_{n-1}}{q_{n-1}}\big|
\le h+
\big|c_0x_0-\frac{a_{n-1}}{q_{n-1}}\big|
\le 
2\big|c_0x_0-\frac{a_{n-1}}{q_{n-1}}\big|
<
\frac{1}{q_{n-1}^{2}}
\]
prove that we can take $q=q_{n-1}$
in \eqref{weyl_ineq}. 
Noting $q_{n-1}\le h^{-1/2}$, it gives
$\beta_f(x_0)\ge
(\alpha-1)/k+2^{1-k}\min(1/2(r-1),1/k)$.
\end{proof}

\begin{proof}[Proof of Proposition~\ref{pointw}]
Let $\mathcal Z$ be the set of zeros of $P'$ modulo $p$. We have $\mathcal Z=\mathcal{B}\cup \mathcal{C}$ where $\mathcal{C}$ is
the set of zeros with multiplicity less than $\nu_F$. We write
\[
 F_z=
\sum_{\substack{n=1\\ n\equiv z\ (p)}}^{\infty}
e\Big(\frac{aP(n)}{q}\Big)
\frac{e\big(P(n)h\big)-1}{n^\alpha}.
\]
Then $F^*(a/q+h)-F^*(a/q)=\sum_{z\in\mathcal Z}F_z$ and the proposition follows if we prove for
some $\gamma>0$
\begin{equation}\label{pointwb}
F_z=p^{-1}A(c_0h)^{(\alpha-1)/k}e\Big(\frac{aP(z)}{q}\Big)
+O\big(p^{-1}h^{(1+\gamma)(\alpha-1)/k}\big)
\quad\text{when } z\in \mathcal{B}
\end{equation}
and for all $1\le a\le q$ but at most $q^{1-\gamma}$ exceptions
\begin{equation}\label{pointwc}
F_z=
O\big(p^{-1}h^{(1+\gamma)(\alpha-1)/k}\big)
\quad\text{when } z\in \mathcal{C}.
\end{equation}

\

If $z\in \mathcal{B}$,  we have 
$e\big(P(z+pn)/q\big)=e\big(P(z)/q\big)$ because $q=p^{\nu_F+1}$ and $z$ is a zero of $P'$ of
multiplicity $\nu_F$. Hence by Euler-Maclaurin summation formula in the form 
$\sum_{n=1}^\infty f(n)=\frac 12 f(1)+\int_1^\infty f(x)\; dx+
O\big(\int_1^\infty |f'(x)|\; dx\big)$ we have
\[
F_z
=
e\Big(\frac{aP(n)}{q}\Big)
\int_{1}^\infty
\frac{e\big(P(z+px)h\big)-1}{(z+px)^\alpha}\; dx
+O\big(h^{\alpha/k}\big)
\]
and Lemma~\ref{integral} gives \eqref{pointwb} after a change of variables for $\gamma$ small enough, since $p^{-1}>h^{1/2k}$.

\

If $z\in \mathcal{C}$, let $m$ be the multiplicity of $z$ in $P'$.  
Then $p^{m-j+1}$ divides $P^{(j)}(z)/j!$ for $0\le j\le m$ and 
the Taylor expansion of $P$ gives that $Q(n)=p^{-m-1}\big(P(z+pn)-P(z)\big)\in\Z[x]$ and
\[
F_z
=
e\Big(\frac{aP(z)}{q}\Big)
\sum_{n=1}^\infty
e\Big(\frac{aQ(n)}{p^{\nu_F-m}}\Big)
\frac{e\big(P(z+pn)h\big)-1}{(z+pn)^\alpha}.
\]
Note that outside the interval
$p^{-1}h^{-1/k+\gamma}<n<p^{-1}h^{-1/k-\gamma}$ 
the contribution of the sum is negligible to prove \eqref{pointwc}. In this interval we seek
cancellation caused by the oscillation of $e\big({aQ(n)}/{p^{\nu_F-m}}\big)$.

By Proposition~4.3 of \cite{gretao}, we have that there exists $K>0$ such that for any $0<\delta
<1/2$ we have
\begin{equation}\label{good_weyl}
\sum_{n\le N} e\Big(\frac{aQ(n)}{p^{\nu_F-m}}\Big)\ll \delta N \qquad \text{ for every } N\ge N_0
\end{equation}
unless
\begin{equation}\label{bad_weyl}
\Big\langle
\frac{aj P^{(m+1)}(z)}{p^{\nu_F-m}}
\Big\rangle
\ll \delta^{-K} N_0^{-m-1}
\qquad\text{for some }1\le j\le \delta^{-K}.
\end{equation}
This is a quantitative form of Weyl's criterion more explicit than usual. Note that for $g=Q$ in
Definition~2.7 of \cite{gretao}, $\alpha_{m+1}=P^{(m+1)}(z)\ne 0$.

\smallskip

Whenever we have (\ref{good_weyl}) for $\delta=h^{(k+1)\gamma}$ and $N_0=p^{-1}h^{-1/k+\gamma}$, summation by parts in the remaining range for $n$ gives
\[
 F_z\ll p^{-1} h^{(1+\gamma)(\alpha-1)/k} + \delta p^{-1}h^{(\alpha-1)/k} h^{-\gamma k} \ll  p^{-1} h^{(1+\gamma)(\alpha-1)/k}.
\]
It only remains to prove that (\ref{bad_weyl}) can be satisfied by very few values of $a$. For a fixed $j$, the number of $a$'s satisfying it is $O(p^{m}+\delta^{-K}N_0^{-m-1}q)$, which is bounded by $q O(p^{-2}+\delta^{-K}N_0^{-2})$. Thus, the contribution from all $j$ is at most $\delta^{-K} q O(p^{-2} +\delta^{-K}N_0^{-2})$ which for $\gamma$ small enough is bounded by $q^{1-\gamma}$, since $p^{-1}>h^{1/2k}$.
\end{proof}

\section{Average oscillation}

The current knowledge about exponential sums and diophantine approximation beyond the quadratic case,
only
allows to control the oscillation of~$F$ in very thin intervals around a given rational. 
For instance, recalling $\tau_0\ll q^{1/2}$, the main term in Proposition~\ref{asymp} does not
become apparent for $h>q^{-k}$, and if $h>q^{-k/2}$ the error term is worse than the trivial
estimate 
$O\big(h^{(\alpha-1)/k}\big)$, cf. \eqref{decompF}.
We
overcome this difficulty  stating results about the
oscillation near most of the rationals having a fixed prime power as denominator. 

Along this section, 
$p$ and
$\tilde{p}$ denote prime numbers not dividing the leading coefficient of $P$ and we write
\begin{equation}\label{q_qtilde} 
 q=\begin{cases}
    p&\text{ if }\nu_F=1
    \\
    p^{\nu_F+1}&\text{ if }\nu_F>1
   \end{cases}
\qquad\text{and} \qquad
 \tilde{q}=\begin{cases}
    \tilde{p}&\text{ if }\nu_F=1
    \\
    \tilde{p}^{\nu_F+1}&\text{ if }\nu_F>1
   \end{cases}
\end{equation}

We assume that the ranges of $\alpha$ and $k$ are as in
Theorem~\ref{mainth},
although the results of this section also apply in the wider range $k/2+1<\alpha<k+1/2$ except
Proposition~\ref{pointw}.

\smallskip

Firstly we show, roughly speaking,  that in some ranges the H\"older exponent is not less than
$(2\alpha-1)/2k$ in most of them.

\begin{proposition}\label{osc1}
For $0<H<1$, we have
\[
\frac 1q \sum_{a=1}^q
\sup_{H\le |h|<2H}
\Big|F_*\big(\frac aq+h\big)-F_*\big(\frac aq\big)\Big|^2\ll H^{(2\alpha-1)/k}+q^{-1} H^{(2\alpha-2)/k}
\]
where $F_*$ means the sum in \eqref{deff} restricted to $\{n\;:\; p\nmid P'(n)\}$ if $\nu_F>1$ and
$F_*=F$ if $\nu_F=1$.
\end{proposition}

No restrictions are needed with an extra averaging in $q$. Even the special form of $q$ (a prime
power) is irrelevant. 

\begin{proposition}\label{osc2}
For $Q^{-k}<H<1$, we have
\[
\frac{1}{Q^{2+\epsilon}}
\sum_{Q\le n< 2Q}
\sum_{a=1}^n
\sup_{H\le |h|<2H}
\Big|F\big(\frac an+h\big)-F\big(\frac an\big)\Big|^2\ll H^{(2\alpha-1)/k}
\]
for any $\epsilon>0$.
\end{proposition}

As we mentioned in the introduction, for the proof of Theorem~\ref{mainth}, integration (mimicked
by global average results like Proposition~\ref{osc1} and Proposition~\ref{osc2}) is too coarse. We
need to perform some local averaging  related to diophantine approximation. Given a denominator
$\tilde{q}$, we are led to consider the fractions $a/q$ with a bigger denominator~$q$ that are very
close to a fraction of denominator $\tilde{q}$. We  state below some average results over this
thin set of fractions~$a/q$. 

For $\tilde{q}<q$ and $t\ge 1$, consider the set 
\[
 A(t)=\big\{1\le a<q\;:\; \langle a\tilde{q}/q\rangle<\tilde{q}t^{-1}\big\}
\]
where $\langle \cdot\rangle$ denotes the distance to the nearest integer. 
Note that $|A(t)|\asymp {q\tilde{q}}/t $ because $|A(t)|$ is counting solutions of
$|a/q-\tilde{a}/\tilde{q}|<t^{-1}$.

\begin{proposition}\label{osc3}
Given  $r>2k/(\nu_F+1)$,  for
$q^{-1/4}<H<\tilde{q}^{-r}$, we have
\[
\frac 1{|A(\tilde{q}^{r})|} \sum_{a\in A(\tilde{q}^{r})}
\sup_{H\le |h|<2H}
\Big|F\big(\frac aq+h\big)-F\big(\frac aq\big)\Big|^2\ll H^{2(\alpha-1)/k+2/(\nu_F+1)r}.
\]
\end{proposition}

\medskip

The conclusion of this analysis of the oscillation is a kind of Chebyshev's inequality with deep
details that is fundamental in the proof of Theorem~\ref{mainth}.

\begin{proposition}\label{cheby}
With the notation of Proposition~\ref{osc3}, consider for $\epsilon>0$
\[
C_{q,\widetilde{q}}(\epsilon)
=\bigg\{
a\in A(\tilde{q}^{r})
\;:\;
\sup_{\frac 14 q^{-r}\le |h|<2\widetilde{q}^{-r}}
\frac{\big|F(a/q+h)-F(a/q)\big|}{|h|^{(\alpha-1)/k+1/(\nu_F+1)r-\epsilon}}
>1
\bigg\}.
\]
Then 
$\big|C_{q,\widetilde{q}}(\epsilon)\big|\ll (\tilde q^r)^{-2\epsilon} |A(\tilde{q}^{r})|\log q$
uniformly for $q\ge \widetilde{q}^{2r(2+\epsilon^{-1})}$.
\end{proposition}

\medskip

For the proofs we employ a kind of maximal theorems for smoothed Weyl sums.

\begin{lemma}\label{osc1L}
Consider the sum
\[
 S(a,M)=
\sideset{}{'}\sum_{n\asymp N}
c_n\eta\big(\frac{P(n)}{M}\big)e\big(\frac{aP(n)}{q}\big)
\]
where $\eta\in C_0^\infty(\R^+)$, $|c_n|\le 1$ and the summation is restricted to $\{n\;:\; p\nmid
P'(n)\}$ if $\nu_F>1$. Then
\[
\frac 1q \sum_{a=1}^q
\sup_{M\le M^*<2M}
 |S(a,M^*)|^2
\ll \|\widetilde{\eta}\|_1^2\big(N+\frac{N^2}{q}\big)
\]
where $\widetilde{\eta}(t)=\int \eta(x)x^{it-1}\; dx$.
The result still holds for unrestricted summation in $n$ and arbitrary $q$ (not necessarily a power
of prime) if the right hand side is multiplied by the multiplicative function $g(q)$ with
$g(p^i)=kp^{i-1}$.
\end{lemma}

\begin{proof}
By Mellin's inversion formula 
\begin{eqnarray*}
|S(a,M)|&=&
\frac{1}{2\pi}
\Big|
\int 
\widetilde{\eta}(t) 
\sideset{}{'}\sum_{n\asymp N}
c_n\big(\frac{P(n)}{M}\big)^{it}e\big(\frac{aP(n)}{q}\big)
dt
\Big|
\\
&\le &
\int 
|\widetilde{\eta}(t)| 
\Big|\sideset{}{'}\sum_{n\asymp N}
c_n\big(P(n)\big)^{it}e\big(\frac{aP(n)}{q}\big)\Big| 
dt.
\end{eqnarray*}
Note that this bound does not depend on $M$. We assume tacitly in the following that $M=M(a)$ is
the $M^*$ giving the supremum. Multiplying both sides by $|S(a,M)|$ and summing on $a$
\begin{equation}\label{sam2}
\sum_{a=1}^q
 |S(a,M)|^2
\le 
\frac{1}{2\pi}
\|\widetilde{\eta}\|_1
\sum_{a=1}^q
\Big|\sideset{}{'}\sum_{n\asymp N}
c_n\big(P(n)\big)^{it_0}e\big(\frac{aP(n)}{q}\big)\Big|
 |S(a,M)|
\end{equation}
for some $t_0\in \R$. 

Note that opening the square and changing the order of summation the inequality 
\begin{equation}\label{sam2aux}
 \sum_{a=1}^q
\Big|\sideset{}{'}\sum_{n\asymp N}
A_ne\big(\frac{aP(n)}{q}\big)\Big|^2
\ll (q+N)N
\end{equation}
holds for any $|A_n|\le 1$ because 
for each fixed $m$, the polynomial congruence $P(n)-P(m)\equiv 0\pmod{p^{\nu_F}}$ has at most $\deg
P$ solutions by Lagrange's Theorem and Hensel's Lemma.

Using \eqref{sam2aux} in \eqref{sam2} after Cauchy-Schwarz inequality, we get the result.
For the last claim in the statement, note that even if the hypotheses of Hensel's Lemma are not
fulfilled, the number of solutions of a polynomial equation $Q(n)\equiv 0\pmod{p^{i}}$ is at most
$p^{i-1}\deg Q$ (see a more detailed than usual statement of Hensel's Lemma in \cite[Ch.6]{moll}). 
\end{proof}

\begin{lemma}\label{osc2L}
With the notation of Lemma~\ref{osc1L} but with no restriction in the summation $S(a,M)$, 
and assuming $\tilde{q}<t<q/2$ and $N^k=o(q)$, we have
\[
\frac{1}{|A(t)|} \sum_{a\in A(t)}
\sup_{M\le M^*<2M}
 |S(a,M^*)|^2
\ll \|\widetilde{\eta}\|_1^2
\mathop{\sum\ \sum}_{
\substack{ n\in\mathcal C_N\ m\in\mathcal C_N\\ |P(n)-P(m)|\le t \\ P(n)\equiv P(m)\pmod{\tilde{q}}}
} 1
\]
where $\mathcal C_N=\{n\asymp N\;:\; c_n\ne 0\}$.
\end{lemma}

\begin{proof}
The first steps in the proof of Lemma~\ref{osc1L} apply and we get \eqref{sam2} with the
new summation range for $a$, and Cauchy-Schwarz inequality gives 
\[
\sum_{a\in A(t)}
 |S(a,M)|^2
\ll
\|\widetilde{\eta}\|_1^2
\sum_{a\in A(t)}
\Big|\sum_{n\asymp N}
c_n\big(P(n)\big)^{it_0}e\big(\frac{aP(n)}{q}\big)\Big|^2.
\]

Let $K_L$ the Fej\'{e}r kernel $L^{-1}\big( \sin(\pi Lx)/\sin(\pi x)\big)^2$ for
$L=[t/\tilde{q}]$. Then
\begin{eqnarray*}
\sum_{a\in A(t)}
\Big|\sum_{n\asymp N}
\Big|^2 
&\ll& 
\frac 1L
\sum_{a=1}^qK_L\big(\frac{a\tilde{q}}{q}\big)
\Big|\sum_{n\asymp N}
\Big|^2
\\ 
&=&
\frac 1L
\sum_j
\big(1-\frac{|j|}{L}\big)_+
\sum_{a\in A(t)}e\big(
\frac{a\tilde{q}}{q}
j\big)
\Big|\sum_{n\asymp N}
\Big|^2.
\end{eqnarray*}
Opening the square and interchanging the order of summation, we get
\[
 \frac{q\tilde{q}}{t}
\sum_{n\in\mathcal C_N}
\sum_{m\in\mathcal C_N}
\#\big\{j\;:\; |j|<t/\tilde{q}, \ P(m)-P(n)+\tilde{q}j\equiv 0\pmod{q}\big\}.
\]
On our assumptions, $|P(m)-P(n)|<q/2$ and $t<q/2$, then the inner set contains at most an element,
given by $j=\big(P(n)-P(m)\big)/\tilde{q}$ which requires $\tilde{q}\mid  P(n)-P(m)$ and
$|P(n)-P(m)|\le t$. 
\end{proof}

For the application of these lemmas in the proof of the previous propositions we need to estimate
$\|\widetilde{\eta}\|_1$ for a special choice of $\eta$.

\begin{lemma}\label{etatilde}
For $\lambda>0$ and  $\phi$ as in Lemma~\ref{d_s_s}, consider
\[
 \eta_\lambda(x)=\phi\big(\frac{x}{\lambda}\big)\big(e(x)-1\big).
\]
Then $\|\widetilde{\eta}_\lambda\|_1\ll \min\big(\lambda, \lambda^{1/2}\big)$.
\end{lemma}

\begin{proof}
After a change of variables
\[
 |\widetilde{\eta}_\lambda(t)|=
\Big|\int \phi(x)\big(e(\lambda x)-1\big)e\big(\frac{t\log x}{2\pi}\big)\; dx\Big|.
\]
By the mean value theorem $e(\lambda x)-1$ can be substituted by $2\pi i x e(\xi_tx)$, $0\le
\xi_t\le \lambda$, and if $0<\lambda \le 1$, integrating by parts twice
$|\widetilde{\eta}_\lambda(t)|\ll \lambda (1+|t|)^{-2}$ and $\|\widetilde{\eta}\|_1\ll
\lambda$.

If $K^{-1}\lambda<t<K\lambda$ for some large constant $K$, by the second derivative test
$|\widetilde{\eta}_\lambda(t)|\ll t^{-1/2}$ and outside this interval, integrating by parts twice
$|\widetilde{\eta}_\lambda(t)|\ll (1+|t|)^{-2}$. Then $\|\widetilde{\eta}\|_1\ll
\lambda^{1/2}$.
\end{proof}

\smallskip

\begin{proof}[Proof of Proposition~\ref{osc1}]
By Lemma~\ref{d_s_s}
\begin{equation}\label{subdya}
F_*\big(\frac aq +h\big)-F_*\big(\frac aq\big)
=
\sum_{L=2^j}
\sideset{}{'}\sum_{n}
\phi\big(\frac{P(n)}{L}\big)\frac{e(h P(n))-1}{n^{\alpha}}e\big(\frac{a P(n)}q\big).
\end{equation}
Then 
\begin{equation}\label{decs1s2}
\frac 1q \sum_{a=1}^q
\sup_{H\le |h|<2H}
\Big|F_*\big(\frac aq+h\big)-F_*\big(\frac aq\big)\Big|^2
\ll
\Sigma_1+\Sigma_2
\end{equation}
where $\Sigma_1$ and $\Sigma_2$ are the contributions coming from the terms $L\le H^{-1}$ and
$L>H^{-1}$ in \eqref{subdya}, respectively.

For each $L$ take
\[
 \eta(x)=K\phi\big(\frac{x}{|h|L}\big)\frac{e(\pm x)-1}{L^{\alpha/k}}
\]
with $K>0$ a large constant and $\pm$ the sign of $h$. With the notation of Lemma~\ref{osc1L},
taking $N\asymp L^{1/k}$ and $c_n=K^{-1}L^{\alpha/k}n^{-\alpha}<1$ we have
\begin{equation}\label{sah}
\sideset{}{'}\sum_{n}
 \phi\big(\frac{P(n)}{L}\big)\frac{e(h P(n))-1}{n^{\alpha}}e\big(\frac{a P(n)}q\big)
=
S(a,|h|^{-1}). 
\end{equation}

By Cauchy-Schwarz inequality, introducing a factor $(LH)^\epsilon\cdot (LH)^{-\epsilon}$, for
$\epsilon>0$
\[
 \Big|
\sum_{L=2^j\le H^{-1}}
\sideset{}{'}\sum_{n}
\Big|^2
\ll
\sum_{L=2^j\le H^{-1}}
(LH)^{-2\epsilon}
\Big|
\sideset{}{'}\sum_{n}
\Big|^2.
\]
Hence
\begin{equation}\label{sigma1} 
 \Sigma_1\ll 
\sum_{L=2^j\le H^{-1}}
(LH)^{-2\epsilon}
\frac 1q \sum_{a=1}^q
\sup_{H\le |h|<2H}
\big|S(a,|h|^{-1})\big|^2.
\end{equation}
Similarly
\begin{equation}\label{sigma2} 
 \Sigma_2\ll 
\sum_{L=2^j> H^{-1}}
(LH)^{2\epsilon}
\frac 1q \sum_{a=1}^q
\sup_{H\le |h|<2H}
\big|S(a,|h|^{-1})\big|^2.
\end{equation}

By Lemma~\ref{osc1L} and Lemma~\ref{etatilde} (note that
$\widetilde{\overline{\eta}}(\xi)=\overline{\widetilde{\eta}(-\xi)}$ and then the uncertainty of the
sign in the definition of $\eta$ is harmless), 
\[
 {\Sigma}_1\ll 
\sum_{L=2^j\le H^{-1}}
(LH)^{-2\epsilon+2} 
\cdot L^{-2\alpha/k}\big(L^{1/k}+L^{2/k}q^{-1}\big)\ll H^{(2\alpha-2)/{k}}(H^{1/k}+q^{-1})
\]
where we have implicitly chosen any $\epsilon<(k-\alpha+\frac 12)/k$. 

In the same way
\[
 {\Sigma_2}\ll 
\sum_{L=2^j> H^{-1}}
(LH)^{2\epsilon} 
HL\cdot L^{-2\alpha/k}\big(L^{1/k}+L^{2/k}q^{-1}\big)\ll H^{(2\alpha-2)/{k}}
(H^{1/k}+q^{-1})
\]
under the assumption $\epsilon<(\alpha-k/2-1)/k$.
\end{proof}

\begin{proof}[Proof of Proposition~\ref{osc2}]
According to the last part of Lemma~\ref{osc1L}, 
we can repeat word by word the proof of Proposition~\ref{osc1} to get 
\[
\frac 1n \sum_{a=1}^n
\sup_{H\le |h|<2H}
\Big|F\big(\frac an+h\big)-F\big(\frac an\big)\Big|^2\ll g(n) H^{(2\alpha-1)/k}.
\]
By \lq\lq Rankin's trick''
\[
\sum_{Q\le n<2Q} ng(n)
\ll Q^{2+\epsilon}\sum_{n=1}^\infty \frac{g(n)}{n^{1+\epsilon}}
\ll Q^{2+\epsilon}\prod_p
\Big(
1
+\frac{k}{p^{1+\epsilon}}
+\frac{k}{p^{1+2\epsilon}}
+\dots
\Big)
\]
and the infinite product converges.
\end{proof}

\begin{proof}[Proof of Proposition~\ref{osc3}]
We start performing the same dyadic subdivision as in \eqref{subdya} 
but in this case replacing $F_*$ by $\tilde{F}_*$ defined analogously but restricting the summation
to $\big\{n\;:\; \tilde{p}\nmid P'(n)\big\}$ when $\nu_F>1$.
Consequently the
sums $S\big(a,|h|^{-1}\big)$ in \eqref{sah} are affected by the same change. In this case we
distinguish three ranges
\[
\frac 1{|A(\tilde{q}^{r})|} \sum_{a\in A(\tilde{q}^{r})}
\sup_{H\le |h|<2H}
\Big|\tilde{F}_*\big(\frac aq+h\big)-\tilde{F}_*\big(\frac aq\big)\Big|^2\ll 
\Sigma_1+\Sigma_2+\Sigma_3
\]
where $\Sigma_1$, $\Sigma_2$ and $\Sigma_3$ correspond to the contributions coming from the terms
with $L\le H^{-1}$, $H^{-1}<L\le q^{k/(k+1)}$ and $L>q^{k/(k+1)}$, respectively. The trivial
estimate for $\Sigma_3$ is
\[
 \Sigma_3\ll 
\big(L^{1/k}\big)^{2(1-\alpha)}
\le q^{2(1-\alpha)/(k+1)}
\le H^{8(\alpha-1)/(k+1)}
\le H^{2(\alpha-1)/k+2/(\nu_F+1)r}.
\]

For $\Sigma_1$ and $\Sigma_2$ we are going to apply Lemma~\ref{osc2L} with $t=\tilde{q}^r$ and
$N\asymp L^{1/k}$. Note that the hypotheses in this lemma are fulfilled. 

Fixed $m$, the condition $P(n)\equiv P(m)\pmod{\tilde{q}}$ gives $O(1+N/\tilde{q})$
solutions for $n$ (by Lagrange's Theorem  and Hensel's Lemma). On the other hand, by the second
condition, any pair of valid values of $n$, say $n_1$ and $n_2$, satisfies $2t>|P(n_1)-P(n_2)|\gg
N^{k-1}|n_1-n_2|$ for $N$ large, then they are concentrated in an interval of length
$O\big(tN^{1-k}\big)$ and, as before, in a fixed number of residue classes determined by $P(n)\equiv P(m)\pmod{\tilde{q}}$.

Summing up, in our case Lemma~\ref{osc2L} reads
\[
\frac{1}{|A(\tilde{q}^r)|} \sum_{a\in A(\tilde{q}^r)}
\sup_{M\le M^*<2M}
\hspace{-5pt}
 |S(a,M^*)|^2
\ll \|\widetilde{\eta}\|_1^2\Big(
L^{1/k}+\tilde{q}^{-1}L^{2/k}\min\big(1,\tilde{q}^rL^{-1} \big)
\Big)
\]
with $\|\widetilde{\eta}\|_1\ll \min\big(HL, (HL)^{1/2}\big)$ by Lemma~\ref{etatilde}.

Therefore
\[
 \Sigma_1\ll 
\sum_{L=2^j\le H^{-1}}
(LH)^{-2\epsilon} 
H^2 L^{2-2\alpha/k}\Big(L^{1/k}+L^{2/k}\tilde{q}^{-1}
\min\big(1,\tilde{q}^r L^{-1}\big)
\Big)
\]
That gives $\Sigma_1\ll H^{2(\alpha-1)/k}\big(H^{1/r}+H^{1/k}\big)$. In the same way, we have
\[
 \Sigma_2\ll 
\sum_{L=2^j> H^{-1}}
(LH)^{2\epsilon} 
H L^{1-2\alpha/k}\big(L^{1/k}+L^{2/k-1}\tilde{q}^{r-1}\big)
\]
and $\Sigma_2\ll H^{2(\alpha-1)/k}\big(H^{1/r}+H^{1/k}\big)$.

This concludes the proof for $\nu_F=1$. For $\nu_F>1$ we still have to prove the bound for
$F-\tilde{F}_*$.
In this case, by definition,
\[
 F(x)-\tilde{F}_*(x)=\sum_{\substack{c=1\\ P'(c)\equiv 0}}^{\tilde{p}}
\sum_{n\equiv c}
\frac{e\big(P(n)x\big)}{n^\alpha}
\]
where the congruences are$\pmod{\tilde{p}}$. Then it is enough to prove
\[
\frac{1}{|A(\tilde{q}^r)|} \sum_{a\in A(\tilde{q}^r)}
\sup_{M\le M^*<2M}
\Big|
\sum_{n\equiv c}\frac{e(h P(n))-1}{n^{\alpha}}e\big(\frac{a P(n)}q\big)
\Big|^2
\ll H^{2(\alpha-1)/k+2/\nu_0r}.
\]
Proceeding as before, after a dyadic subdivision and disregarding the terms $n>q^{1/(k+1)}$ by the
trivial estimate, we have to deal with
\[ 
 \Sigma_1= 
\sum_{L=2^j\le H^{-1}}
(LH)^{-2\epsilon}
\frac{1}{|A(\tilde{q}^r)|} \sum_{a\in A(\tilde{q}^r)}
\sup_{H\le |h|<2H}
\big|S(a,|h|^{-1})\big|^2
\]
and
\[
 \Sigma_2=
\sum_{H^{-1}<L=2^j\le q^{k/(k+1)} }
(LH)^{2\epsilon}
\frac{1}{|A(\tilde{q}^r)|} \sum_{a\in A(\tilde{q}^r)}
\sup_{H\le |h|<2H}
\big|S(a,|h|^{-1})\big|^2
\]
where now $c_n=0$ when $n\not\equiv c\pmod{\tilde{p}}$.

When applying Lemma~\ref{osc2L} with $t=\tilde{q}^r$ and $N\asymp L^{1/k}$,
relaxing$\pmod{\tilde{q}}$  to$\pmod{\tilde{p}}$ we obtain the bound
$O\big(1+L^{2/k}/\tilde{p}^2\big)$ for the double summation. On the other hand, fixing $m$,
$|P(c+l\tilde{p})-P(m)|\le \tilde{q}^r$ only can hold for
$O\big(\tilde{q}^r/\tilde{p}L^{1-1/k}\big)$ and the double summation is
$O\big(\tilde{q}^r/\tilde{p}^2L^{1-2/k}\big)$.

Recalling $\|\widetilde{\eta}\|_1\ll \min\big(HL, (HL)^{1/2}\big)$ by Lemma~\ref{etatilde}, we have 
\[
 \Sigma_1\ll 
\sum_{L=2^j\le H^{-1}}
(LH)^{-2\epsilon} 
H^2 L^{2-2\alpha/k}\Big(1+L^{2/k}\tilde{p}^{-2}\min\big(1,\tilde{q}^rL^{-1}\big)\Big)
\]
that is $\ll H^{2(\alpha-1)/k+2/r\nu_F}$, as expected, using $\tilde{q}^r<H^{-1}$ and
$\tilde{p}=\tilde{q}^{1/(\nu_F+1)}$.
And
\[
 \Sigma_2\ll 
\sum_{L=2^j> H^{-1}}
(LH)^{2\epsilon} 
H L^{1-2\alpha/k}\big(1+L^{-1+2/k}\tilde{p}^{-2}\tilde{q}^r\big)
\ll H^{2(\alpha-1)/k+2/r(\nu_F+1)}.
\]
This concludes the proof.
\end{proof}

\begin{proof}[Proof of Proposition~\ref{cheby}]

For $q^{-1/4}<|h|<2\widetilde{q}^{-r}$, Proposition~\ref{osc3} and Cheb\-yshev's inequality after a subdivision into dyadic intervals (which introduces an extra logarithm), give a contribution $O\big((\tilde q^r)^{-2\epsilon} |A(\tilde q^r)|  \log q\big)$.

For  $\frac 12q^{-r}<|h|\le q^{-1/4}$, the trivial bound when $\nu_F>1$
\[
 |F^*(\frac aq+h)-F^*(\frac aq)|\ll p^{-1}|h|^{\frac{\alpha-1}{k}}\ll |h|^{\frac{\alpha-1}k+\frac{1}{r(\nu_F+1)}}
\]
gives an admissible contribution. It remains to  take care of $F_*$ in the range $\frac 12q^{-r}<|h|\le q^{-1/4}$. The same argument as above with Chebyshev's inequality after a dyadic subdivision, but now using Proposition~\ref{osc1}, gives $O\big((q^{1/4})^{-2\epsilon} q \log q\big)$.

Adding these contributions, we get the expected bound
in the range $q\ge \widetilde{q}^{2r(2+\epsilon^{-1})-2\epsilon^{-1}}$ that is wider that the one in the statement.
\end{proof}

\section{Arithmetic lemmata}

If $q$ is a prime power such that $P'$ has an unusual number of zeros modulo $q$, then the behavior
of $F(a/q+h)-F(a/q)$ depends, in part, on a short exponential sum, as we saw in
Proposition~\ref{pointw}, and a basic task is to prove that
it is large enough for a significant part of the values of $a$ when working in a fixed interval. We
state a general result of this kind here, with no reference to the actual parameters and frequencies
appearing in our particular case.

\begin{lemma}\label{turan}
Let $\{\lambda_n\}_{n=1}^N$ be a list of integers and $q\in\Z^+$. For every closed interval
$I\subset [0,1]$ satisfying $|I|\ge N^2/q$, it holds
\[
\frac{1}{q|I|}
\#\Big\{
a\in\Z\; :\; \frac aq\in I
\text{ and }
\big|
\sum_{n=1}^Ne\big(\frac{\lambda_n a}{q}\big)
\big|
\ge 
\big(
\frac{|I|}{100}
\big)^N
\Big\}
\ge 
\frac 1{(100N)^2}.
\]
\end{lemma}

The proof is based on Tur\'{a}n method as described in \cite{Mon}. In particular, we employ

\begin{theorem}[Tur\'{a}n's First Main Theorem {\cite[\S5.2]{Mon}}]\label{turan1}
Let $s_\nu$ denote the sum
\[
s_\nu=\sum_{n=1}^Nb_nz_n^\nu
\qquad\text{with }
b_n,z_n\in\C
\text{ and }
|z_n|\ge 1.
\]
Then for any non-negative integer $M$ there exists $\nu\in [M+1,M+N]\cap \Z$ such that
\[
|s_\nu|\ge
|s_0|\Big(
\frac{N}{2e(M+N)}
\Big)^{N-1}.
\]
\end{theorem}

We shall also appeal to a result of the same flavor. 

\begin{theorem}[cf. {\cite[\S5.4]{Mon}}]\label{turanm}
With the notation as in the previous result, for any $H\ge N$
\[
\sum_{\nu=1}^H
|s_\nu|^2\ge
\frac{|s_0|^2}{e^{8N^2/H}-1}.
\]
\end{theorem}

Note that this statement is slightly stronger than Th.3 in \cite[\S5.4]{Mon} but this is the result
actually proved there.

\begin{proof}[Proof of Lemma~\ref{turan}]
Consider $s_\nu$ as in Theorem~\ref{turan1} with  $z_n=e(\lambda_n/q)$ and  $b_n=e(\lambda_n
a_0/q)$, choosing $a_0/q\in I$ such that $|s_0|\hspace{-1pt}=\hspace{-1pt}\max \big|\sum_{n=1}^N 
\hspace{-3pt}e(\lambda_n a/q)\big|$ for
$a/q\in I$.

Take $H=|I|q/2$ and note that perhaps changing $z_n$ by $\overline{z}_n$ we have $s_\nu=
\sum_{n=1}^N e(\lambda_n a/q)$  for $a/q\in I$ with $1\le \nu\le H$. Let $\delta$ be the proportion
of these values of $\nu$ satisfying $|s_\nu|\ge |s_0|/86N$. By Theorem~\ref{turanm}, dividing
by~$H|s_0|^2$, we have
\[
\delta+(1-\delta)\frac{1}{(86N)^2}\ge\frac{H^{-1}}{e^{8N^2/H}-1}.
\]
Writing $y=N^2/H<1$ it is easy to derive that $\delta>2(100 N)^{-2}$. Hence there are $H\delta>(100
N)^{-2}|I|q$ values of $\nu$ such that $|s_\nu|\ge |s_0|/86N$ and the result follows if we prove
\begin{equation}\label{ineq_turan}
\frac{|s_0|}{86N}\ge
\big(
\frac{|I|}{1000}
\big)^N.
\end{equation}
This is indeed a simple variation of Lemma 1 in \cite[\S5.3]{Mon}, due to Tur\'{a}n. By completeness
we provide the proof here.

Our assumptions allow to find an interval $I'\subset I$ with $a_0/q\in I'$ such that $|I'|\ge |I|/2$
and $q|I'|/N=m\in\Z$. If $\alpha$ is the lower limit of $I'$ then $\nu m/q\in I'$ for $M+1\le \nu\le
M+N$ with $M=\lfloor \alpha q/m\rfloor$, and $N/(M+N)\ge |I'|$.

By the maximal property of $|s_0|$ and applying Theorem~\ref{turan1} to
$\widetilde{s}_\nu=\sum_{n=1}^Ne(\lambda_n m\nu /q)$, we deduce
\[
|s_0|\ge\max_{M+1\le\nu\le M+N}|\widetilde{s}_\nu|
\ge 
|\widetilde{s}_0|
\big(
\frac{|I'|}{2e}
\big)^{N-1}
\ge 
N
\big(
\frac{|I|}{4e}
\big)^{N-1},
\]
that is stronger than \eqref{ineq_turan}.
\end{proof}

We also need a lower bound for a positive proportion of the $\tau_{0}$ in each interval (cf.
Proposition 4.2 of \cite{chaubi}). It can be that in some cases one can get a stronger result by the methods in \cite{katz}, but we could not locate a precise statement there.

\begin{lemma}\label{lbtau} 
Let $I\subset (0,1)$ be a closed interval.
Given $0<\epsilon<1$ there exist $C=C(\epsilon)$  such that if $Q |I|^{2+\epsilon}>C$
then
\[
\frac{Q^2|I|}{C\log Q}\le 
\#\big\{\frac{a}{q}\in I\;:\; Q\le q<2Q,\text{ $q$ prime and } |\tau_{0}|\ge \frac 12 \sqrt{q}
\big\}
\le C
\frac{Q^2|I|}{\log Q}.
\]
\end{lemma}

\begin{proof}
In \cite{chajim} is proved that $P(x)-P(y)+r$ is absolutely irreducible over $\F_q$ when $q\nmid
\deg P$ and $r\ne 0$. Hence by the Riemann hypothesis over finite fields (see Ch.3 of
\cite{schmidt}) the corresponding algebraic curve contains $q+O\big(q^{1/2}\big)$ points over $\F_q$
and proceeding as in Proposition~4.4 of \cite{chaubi} we have for each~$q$  
\begin{equation}\label{tau_av}
\sum_{a\;:\; a/q\in I}
|\tau_0|^2
=
|I|
\sum_{a=1}^{q-1} |\tau_{0}|^2
+O(q^{3/2}\log q)
\end{equation}
and the $O$-constant only depends on $P$ (see the details in \cite{chajim}).

Divide the values of $a$ in the left hand side of \eqref{tau_av} into two sets
$\mathcal{A}$ and $\mathcal{B}$ according $|\tau_0|\ge \frac 12 \sqrt{q}$ or not, respectively.
Assume 
\begin{equation}\label{tau_av_0}
\sum_{a=1}^{q-1} |\tau_{0}|^2\ge \frac 12 q^2,
\end{equation}
then 
\[
\sum_{a\in \mathcal{A}}
|\tau_0|^2
\ge 
\frac 12 |I|q^2
-\frac 14 q|\mathcal{B}|
+O(q^{3/2}\log q).
\]
By Weil's bound (see Corollary~2.1 in \cite[Ch.2]{schmidt}) $|\tau_0|\le (\deg P-1)q^{1/2}$ for
$q>\deg P$ and this and  the trivial estimate $|\mathcal{B}|\le q|I|+1$, give
$|\mathcal{A}|\gg q|I|$. In other words
(cf.
Proposition~4.2 in \cite{chaubi}), for $q$ large 
\[
q|I|\ll
\#\big\{a\;:\; \frac{a}{q}\in I\text{ and } |\tau_{0}|\ge \frac 12
\sqrt{q}
\big\}\le 
q|I|+1.
\]

Note that opening the square 
\begin{equation}\label{tau_trans}
\sum_{a=1}^{q-1} |\tau_{0}|^2=
q\#\big\{(n,m)\in\F_q\times\F_q\;:\; P(n)-P(m)=0\big\}-q^2.
\end{equation}
Then to deduce that \eqref{tau_av_0} holds for a positive proportion of the primes, the aim is to prove that for many primes the congruence $P(n)\equiv P(m)\pmod{q}$ has many solutions
distinct from the obvious diagonal ones. This is done in  \cite{chajim}. Theorem~1.2 in that paper gives
\[
 \sum_{a=1}^{q-1} |\tau_{0}|^2\ge q^2+O\big(q^{3/2}\big)
\]
if $P$ is not a composition of linear and Dickson polynomials. In the rest of the cases this inequality fails for infinitely many primes~$q$ but Theorem~1.3 in \cite{chajim} proves that it holds for a positive proportion of the primes, determined by their splitting properties in certain number fields. In this way we have that a stronger form of 
\eqref{tau_av_0} holds true for a positive proportion of the primes.
\end{proof}

Finally, we also need to extract a well-spaced subset from the fractions having a prime power as
denominator.

\begin{lemma}\label{spacing} 
Consider the set 
\[
X_0=\big\{\frac{a}{p^n}\in I\;:\; Q\le p^n<2Q,\text{ $p$ prime}\big\}
\]
where $I\subset (0,1)$ is a closed interval with $Q^{1-2/n}|I|>4$ and  $n>2$ is a fixed integer. 
Given $0<\alpha<1$ there exists $0<\beta<1$ (only depending on $\alpha$) such that for any
$X_1\subset X_0$ with  $|X_1|>\alpha |X_0|$ we can find $X_2\subset X_1$ with $|X_2|>\beta |X_0|$
satisfying $|x-y|\ge |I|/|X_0|$ for any $x$ and $y$ distinct elements of $X_2$.
\end{lemma}

\begin{proof}
Let $D$ be the number of primes $p$ in $[Q^{1/n},(2Q)^{1/n})$. It is the number of different
possible denominators of the elements of $X_0$ and $X_0$ is the union of at most $D$ finite
arithmetic progressions. Hence the result is obvious if $D$ is bounded by an absolute constant and
we  can assume $D\ge 6$. In this case
\begin{equation}\label{size_x0}
|X_0|\ge D\big( Q|I|-1\big)>4Q|I|.
\end{equation}

Omit one of the extreme points of $I$ and subdivide the result into $|X_0|$ equal consecutive
half-open intervals $I_1,I_2,\dots, I_{|X_0|}$ of length $|I|/|X_0|$. Consider a subsequence of them
$O_i=I_{n_i}$ defined by the property $|O_i\cap X_0|\le \lambda$.
Assume the following inequality that we shall prove later
\begin{equation}\label{no_cluster}
\#\big\{
x\in X_0\;:\; x\not\in \bigcup O_i
\big\}
\le \frac{8000}{\lambda}|X_0|.
\end{equation}
Taking $\lambda = 16000/\alpha$, we have 
\[
\#\big\{
x\in X_1\;:\; x\in \bigcup O_i
\big\}
> \frac{1}{2}|X_1|,
\]
and $|O_i\cap X_1|\le \lambda=16000/\alpha$ implies that $O_i\cap X_1\not=\emptyset$ for at least
$\alpha^2|X_0|/32000$ values of~$i$. Choose an element in each of these nonempty sets and eliminate
at most one half of them to assure that they do not belong to consecutive intervals $I_j$. Then we
get a set $X_2$ such that $|X_2|>\beta |X_0|$ with $\beta=\alpha^2/64000$ and $|x-y|\ge |I|/|X_0|$
for any $x,y\in X_2$, $x\ne y$.

\

It only remains to prove \eqref{no_cluster}. Consider the function
\[
f(x)=\#\Big\{
p^n\in [Q,2Q)\; :\; \langle p^nx\rangle <\frac{2|I|Q}{|X_0|}
\Big\}
\]
where $\langle \cdot\rangle$ denotes the distance to the nearest integer. 
By \eqref{size_x0} the elements of $|I_i\cap X_0|$ have different denominators
(because $|I_i|=|I|/|X_0|<p^{-n}$).
If $\Lambda \le |I_i\cap X_0|<2\Lambda$ then $f(t)\ge \Lambda$ for every $t\in I_i$. Hence the
number of intervals $I_i$ with $|I_i\cap X_0|\in [\Lambda, 2\Lambda)$ is at most 
$\Lambda^{-2}|I|^{-1}|X_0|\int _I f^2$ 
and we have
\begin{equation}\label{no_cluster2}
\#\big\{
x\in X_0\;:\; x\not\in \bigcup O_i
\big\}
\le \frac{4|X_0|}{|I|\lambda}\int _I f^2.
\end{equation}

Let $L$ be the integral part of $|X_0|/(4|I|Q)$, note that $L\ge 1$ by \eqref{size_x0},
then it is easy to see that
\[
f(x)\le \frac{\pi^2}{4}
\sum_{Q\le p^n<2Q}
\frac{\sin^2(\pi Lp^nx)}{L^2\sin^2(\pi p^nx)}
=
\frac{\pi^2}{4L}
\sum_{Q\le p^n<2Q}
\sum_{|k|\le L}
\big(1-\frac{|k|}{L}\big)
e(p^nkx).
\]
Substituting in \eqref{no_cluster2}, by Lemma~7.1 of \cite{IwKo} (with $Y=|I|/2$ and a translation), which is essentially to add a Fej\'er kernel to the integral,
we have
\[
\#\big\{
x\in X_0\;:\; x\not\in \bigcup O_i
\big\}
\le \frac{61|X_0|}{\lambda L^2}
\mathop{\sum\sum\ \ \ \sum\sum}_{
\substack{ |k_1|,|k_2|\le L\ \ Q\le p_1^n,p_2^n<2Q\\ |p_1^nk_1-p_2^nk_2|<|I|^{-1} }
} 
\big(1-\frac{|k_1|}{L}\big)
\big(1-\frac{|k_2|}{L}\big).
\]
The terms with $k_1=k_2=0$ contribute to the sum as $D^2$.
If $k_1=k_2\ne 0$, necessarily $p_1^n=p_2^n$ and the contribution of the sums is $DL$.
Finally in the rest of the cases, for each $(k_1,k_2)$ there is at most a pair $(p_1^n,p_2^n)$
satisfying $|p_1^nk_1-p_2^nk_2|<|I|^{-1}$ because a second pair $(p_3^n,p_4^n)$ would give, assuming
without loss of generality $k_1\ne 0$ and $p_1>p_2$,
\[
\frac{2|I|^{-1}}{Q}
>
\Big|
\frac{p_1^n}{p_2^n}-\frac{p_3^n}{p_4^n}
\Big|
>
\Big|
\frac{p_1}{p_2}-\frac{p_3}{p_4}
\Big|
\ge
\frac{1}{(2Q)^{2/n}}
\]
that contradicts our assumption $Q^{1-2/n}|I|>4$. Hence we can forget the summations over $p_1^n$
and $p_2^n$, getting that the multiple sum is at most $L^2$.

Summing up, we have
\[
\#\big\{
x\in X_0\;:\; x\not\in \bigcup O_i
\big\}
\le \frac{61|X_0|}{\lambda L^2}
\Big(
D^2+DL+L^2
\Big).
\]
Using the definition of $L$ and the first inequality in \eqref{size_x0}, it follows
\[
 L>\frac{|X_0|}{8|I|Q}
\ge D\frac{|I|Q-1}{8|I|Q}
>\frac{3D}{32}.
\]
Substituting in the previous formula, we obtain \eqref{no_cluster}.
\end{proof}

\section{Some Cantor-like sets}

For the lower bound in the main theorem we need to compute a lower bound for the Hausdorff dimension
of some limit sets of sequences of nested intervals.

In general, if $\mathcal{G}_1, \mathcal{G}_2, \mathcal{G}_3,\dots$ are sets of
disjoint closed intervals in $[0,1]$, we say that $I\in \mathcal{G}_i$ is the \emph{son} of $J\in
\mathcal{G}_{i-1}$ if $I\subset J$. And we say that $I_1,I_2\in \mathcal{G}_i$ are \emph{brothers}
if they are sons of the same interval $J\in \mathcal{G}_{i-1}$. By convenience we define
$\mathcal{G}_0=\{[0,1]\}$ and then all the intervals of $\mathcal{G}_1$ are brothers.

In the usual Cantor set (and in some straightforward variants) the number of sons of each interval
is fixed. The rough idea is that more sons  imply more dimension and then a number of sons
bounded from below is enough to get a lower bound for the Hausdorff dimension.

\begin{lemma}\label{cantor}
Let $\mathcal{G}_i$, $i\in\Z^+$, be nonempty sets of  disjoint intervals
in $[0,1]$ and $G_i=\bigcup_{I\in \mathcal{G}_i} I$. If $G_1\supset G_2\supset G_3\supset\dots$ and
any $I\in \mathcal{G}_i$ has at least $m_i>1$ brothers separated by gaps of at least $\delta_i$
with $\{\delta_i\}_{i=1}^\infty$ strictly decreasing to~$0$, then
$G=\bigcap_{i=1}^\infty G_i$ satisfies
\[
 \dim_HG\ge
\liminf_{n\to \infty}
\frac{\sum_{i=1}^{n-1}\log m_i}{-\log(m_{n}\delta_{n})}.
\]
On the other hand, if any interval in $\mathcal{G}_i$ has at most $M_i>1$ brothers and length at
most $L_i$ then
\[
 \dim_HG\le
\limsup_{n\to \infty}
\frac{\sum_{i=1}^{n}\log M_i}{-\log L_{n}}.
\]
\end{lemma}

\begin{proof}
The first part is the Example 4.6 of \cite{Fal}. The second part follows from the definition of
Hausdorff measure because
\[
\sum_{I\in\mathcal{G}_n}|I|^s\le M_1M_2\cdots M_n L_n^s
\]
and if $s$ is greater than the upper limit the sum goes to~$0$ as $n\to \infty$.
\end{proof}

Note that for instance, $m_i=M_i=2$ and $L_i=3^{-i}$  for the usual Cantor set $C$, and
Lemma~\ref{cantor} implies $\dim_HC=\log_32$. 

In our construction some of the intervals (infinitely many of them) are banned. The idea is that
if we preserve a substantial part of the intervals  we should still keep the lower bound while the 
upper bound follows by inclusion.

\begin{lemma}\label{cantor2}
Let $\mathcal{G}_i$ be as in Lemma~\ref{cantor} but now we assume that any  $I\in
\mathcal{G}_i$ has at least $2m_i$ brothers.
Let $\mathcal{G}_i^*$ be as $\mathcal{G}_i$ but omitting $d_i$ intervals and define accordingly
$G_i^*=\bigcup_{I\in \mathcal{G}_i^*} I$.
If the cardinality of $\mathcal{G}_1$ is greater than $2d_1+2\sum_{i=2}^\infty d_i\prod_{j=2}^i
m_i^{-1}$ then
$G^*=\bigcap_{i=1}^\infty G_i^*\ne\emptyset$ and
\[
 \dim_HG^*\ge
\liminf_{n\to \infty}
\frac{\sum_{i=1}^{n-1}\log m_i}{-\log(m_{n}\delta_{n})}.
\]
\end{lemma}

\begin{proof}
Let $\mathcal{G}_i'$ as $\mathcal{G}_i^*$ but omitting the intervals having less than $m_i$
brothers. If $I\in\mathcal{G}_i^*\setminus \mathcal{G}_i'$, more than one half of its brothers in
$\mathcal{G}_i$ are not in $\mathcal{G}_i^*$, hence 
$|\mathcal{G}_i^*\setminus \mathcal{G}_i'|\le |\mathcal{G}_i\setminus \mathcal{G}_i^*|$, in
particular $|\mathcal{G}_i'|\ge 2|\mathcal{G}_i^*|-|\mathcal{G}_i|\ge |\mathcal{G}_i|- 2d_i$ and
$\mathcal{G}_1'\ne \emptyset$ under the hypothesis of the lemma.
Clearly any interval in  $\mathcal{G}_i'$ has either at least $m_i$ sons or none. This latter
possibility impedes a direct application of Lemma~\ref{cantor}.

Let $\Delta_i=\big\{I\in \mathcal{G}_1'\; :\; \exists J\in\mathcal{G}_i'
\text{ with no sons and }J\subset I\big\}$.
If $\mathcal{A}=\bigcap_{i=1}^\infty \big(\mathcal{G}_1'\setminus\Delta_i\big)$ is non empty then
the sets $\big\{J\in \mathcal{G}_i'\; :\; J\subset I\in \mathcal{A}\big\}$ are also non empty
(note that $\mathcal{G}_i'=\emptyset$ implies $\Delta_{i-1}=\mathcal{G}_1'$) and Lemma~\ref{cantor}
gives the result.

It remains to show that $\mathcal{A}=\emptyset$ leads to a contradiction.  If
$\mathcal{A}=\emptyset$ then we can assign to each $I\in \mathcal{G}_1'$ the smallest index $i=i(I)$
such that $I\in\Delta_i$.
There are at most $|\mathcal{G}_2\setminus\mathcal{G}_2'|/m_2\le 2d_2/m_2$ intervals with $i(I)=1$
because each interval in $\mathcal{G}_1'\setminus\Delta_1$ has at least $m_2$ sons. In the same way,
there are at most $|\mathcal{G}_3\setminus\mathcal{G}_3'|/m_2m_3\le 2d_3/m_2m_3$ intervals with
$i(I)=2$ because each interval in $\big(\mathcal{G}_1'\setminus\Delta_1\big)\cap
\big(\mathcal{G}_1'\setminus\Delta_2\big)$ has at least $m_2m_3$ grandsons (sons of sons). In
general, considering all possible values of $i(I)$, we have
\[
|\mathcal{G}_1'|\le 
\frac{2d_2}{m_2}+\frac{2d_3}{m_2m_3}+\frac{2d_4}{m_2m_3m_4}+\dots
\]
which contradicts our hypothesis (recall that $|\mathcal{G}_1'|\ge |\mathcal{G}_1|-2d_i$).
\end{proof}

\section{Proof of the main theorem}

The lower bound for the spectrum of singularities comes from the construction of a set with
Hausdorff dimension $(\nu_0+2)/2r$ and $\beta_F(x)=(\alpha-1)/k+1/2r$ for every~$x$ in this set.

The proof is clearer  in the case $\nu_F=1$ that we present firstly.

\begin{proposition}\label{set_unram}
With the notation in main theorem, if $\nu_F=1$,  given any $r>k$ there exists a set $C_r$
such that $\dim_HC_r=2/r$ and $\beta_F(x)=(\alpha-1)/k+1/2r$ for every $x\in C_r$.
\end{proposition}

\begin{proof}
For each irreducible fraction $a/q$ we consider the closed interval
\[
 I_{a/q}=\Big[\frac{a}{q}+\frac{1}{4q^r}, \frac{a}{q}+\frac{1}{2q^r}\Big].
\]
In our case $\nu_F=1$ we assume that $q$ is prime (recall the comments at the beginning of \S3).
Note that this intervals are disjoint.

We are going to construct nonempty sets of  disjoint
intervals $\mathcal{G}_i$ fulfilling the hypothesis in Lemma~\ref{cantor2}.
We take $\mathcal{G}_0=\{[0,1]\}$ and 
\[
\mathcal{G}_i=\big\{
I_{a/q}\subset J\in \mathcal{G}_{i-1}\;:\; Q_i\le q< 2Q_i,\ |\tau_{0}|\ge \frac 12\sqrt{q}
\big\},
\qquad 
i=1,2,3,\dots
\]
where $Q_i=K^{(Ri)!}$ with $R=\lceil 5r\rceil$ and $K$ a large enough constant.
Note that with the notation introduced before Proposition~\ref{osc3} if 
$I_{a/q}\subset I_{\widetilde{a}/\widetilde{q}}$
with 
$I_{a/q}\in\mathcal{G}_i$ and $I_{\widetilde{a}/\widetilde{q}}\in\mathcal{G}_{i-1}$
then $a\in A\big(\widetilde{q}^r\big)$.

The brothers of a given $I_{a/q}\in\mathcal{G}_i$ are spaced at least $\delta_i=\frac 18Q_i^{-2}$
because $|a/q-a'/q'|>1/qq'$. And the number of brothers $m_i$ satisfies
\begin{equation}\label{brothers}
\frac{Q_i^2Q_{i-1}^{-r}}{C\log Q_i}
<m_i<
C\frac{Q_i^2Q_{i-1}^{-r}}{\log Q_i}
\end{equation}
for certain $C>0$, just applying Lemma~\ref{lbtau} with $I$ the interval of $\mathcal{G}_{i-1}$
containing all the brothers of $I_{a/q}$.

Now we consider $\mathcal{G}_1^*=\mathcal{G}_1$, $\mathcal{G}_2^*=\mathcal{G}_2$ and for $i>2$, with
the notation of Proposition~\ref{cheby},
\[
\mathcal{G}_i^*=
\Big\{
I_{a/q}\in \mathcal{G}_{i}\;:\; I_{a/q}\subset I_{\widetilde{a}/\widetilde{q}} \in
\mathcal{G}_{i-1}\text{ with }
a\not\in C_{q,\widetilde{q}}(4/i)
\Big\}.
\]
The definition of $Q_i$ implies $Q_i\ge (2Q_i)^{2r(2+i/4)}$ as required in Proposition~\ref{cheby} for $Q_i\le q<2Q_i$ and $Q_{i-1}\le \widetilde{q}<2Q_{i-1}$. Then
for $d_i=|\mathcal{G}_{i}|-|\mathcal{G}_{i}^*|$ we have
\[
d_i\ll Q_{i-1}^{-8/i}(\log Q_i)
\sum_{q\asymp Q_i}
\sum_{\widetilde{q}\asymp Q_{i-1}}
\big|A\big(\widetilde{q}^r\big)\big|
\ll
Q_{i-1}^{-8/i}Q_{i-1}^{2-r}Q_i^2(\log Q_{i-1})^{-1}.
\]
Note that the cardinality of $A\big(\widetilde{q}^r\big)$ is $O\big(\widetilde{q}^{1-r}q\big)$
because fixed one of the $O(\widetilde{q})$ possible values of the integral part of
$a\widetilde{q}/q$, there are $O\big(\widetilde{q}^{-r}q\big)$ possibilities for $a$.

Hence \eqref{brothers} implies
\[
\frac{d_i}{m_2m_3\cdots m_i}
\ll C^i
Q_{i-1}^{-8/i}Q_1^{2}\frac{\log Q_i}{\log Q_1}
\prod_{j=1}^{i-2}
\big(Q_j^{r-2}\log Q_j\big)
\]
and this is $\ll Q_{i-1}^{-1/i}$ and the condition in Lemma~\ref{cantor2} is fulfilled for $K$ large
enough, giving
\[
 \dim_HG^*\ge
\liminf_{n\to \infty}
\frac{\log m_{n-1}}{-\log(m_{n}\delta_{n})}
\ge \frac{\log\big(Q_{n-1}^{2}Q_{n-2}^{-r}\big)}{-\log Q_{n-1}^{-r}}
=\frac 2r.
\]
On the other hand, by Lemma~\ref{cantor}
\[
 \dim_HG^*\le
 \dim_HG\le
\limsup_{n\to \infty}
\frac{\log Q_n^2 }{-\log Q_n^{-r}}
=\frac 2r.
\]

We take $C_r=G^*$. It remains to check $\beta_F(x)=(\alpha-1)/k+1/2r$ for every~$x\in G^*$.

Given $x\in G^*$, take $h_i=x-a_i/q_i>0$ where $x\in I_{a_i/q_i}\in\mathcal{G}_i^*$.
Proposition~\ref{asymp} gives
\[
|F(x-h_i)-F(x)|=\big|F(\frac{a_i}{q_i}+h_i)-F(\frac{a_i}{q_i})\big|
\gg h_i^{(\alpha-1)/k+1/2r}
\]
where we have employed $h_i\asymp q_i^{-r}$ and $|\tau_0|\gg\sqrt{q}$ by the definition of
$\mathcal{G}_i$. Hence,  $\beta_F(x)\le (\alpha-1)/k+1/2r$.

On the other hand, given $x\in G^*$ and $h$ small enough, we can find
$I_{\widetilde{a}/\widetilde{q}}\supset I_{a/q}\ni x$ with $I_{a/q}\in \mathcal{G}_{i}$ and
$I_{\widetilde{a}/\widetilde{q}} \in \mathcal{G}_{i-1}$, and $q^{-r}\le |h|<\widetilde{q}^{-r}$.
Then $\frac{1}{2}q^{-r}\le |x+h-a/q|<2\widetilde{q}^{-r}$ and by the definition of
$\mathcal{G}_i^*$,
\[
|F(x+h)-F(x)|\le 
\big|F(x+h)-F(\frac{a}{q})\big|
+\big|F(x)-F(\frac{a}{q})\big|
\ll |h|^{-4/i} |h|^{(\alpha-1)/k+1/2r}
\]
so letting $i\to\infty$, allows to conclude
$\beta_F(x)\ge (\alpha-1)/k+1/2r$.
\end{proof}

In general terms the proof in the case $\nu_F>1$ parallels the previous proof but there are some
problems when adapting it. Firstly, the spacing of fractions with denominator given by a perfect
power is a difficult issue (even for small exponents, see for instance \cite{zaha}) that we avoid
using Lemma~\ref{spacing} to work in a well-spaced subset. 
Secondly, now the role of $\tau_0$ is, in some way, played by $\tau_*$, a very short exponential sum
depending on the ramification modulo $p$. An algebraic approach seems unfeasible and we instead
appeal to Lemma~\ref{turan}, a variant of Tur\'{a}n method, to get an lower bound.
Finally, a requirement in Proposition~\ref{pointw} to introduce $\tau_*$ is the equality between
\lq\lq local'' and \lq\lq global'' ramification. To fulfill it we appeal to a weak form of
Chebotarev's theorem due to Frobenius \cite{chebotarev} that already appeared indirectly in the
case $\nu_F=1$ through the application of \cite[Th 1.3]{chajim} in the proof of Lemma~\ref{lbtau}.

\begin{proposition}\label{set_ram}
With the notation in main theorem, if $\nu_F>1$,  given any $s>k$ there exists a set $C_s$
such that $\dim_HC_s=(\nu_F+2)/2s$ and $\beta_F(x)=(\alpha-1)/k+1/2s$ for every $x\in C_s$.
\end{proposition}

\begin{proof}
We consider, as before, for each irreducible fraction $a/q$ the closed interval
\[
I_{a/q}=
\Big[\frac{a}{q}+\frac{1}{4q^r}, \frac{a}{q}+\frac{1}{2q^r}\Big]
\qquad\text{with }\quad
r=\frac{2s}{\nu_F+1}.
\]
Recall that now $q=p^{\nu_F+1}$ with $p$ prime, according to the notation introduced in
\eqref{q_qtilde}.

We take $\mathcal{G}_0=\{[0,1]\}$ and 
\[
\mathcal{G}_i=\big\{
I_{a/q}\subset J\in \mathcal{G}_{i-1}\;:\; {a}/{q}\in S_i(J)
\big\},
\qquad 
i=1,2,3,\dots
\]
where $S_i(J)$ is a set of irreducible fractions $a/q\in J$ such that the following conditions are
fulfilled

\begin{enumerate}
\renewcommand{\labelenumi}{(\roman{enumi})} 
\item$Q_i\le q=p^{\nu_F+1}<2Q_i$ where $Q_i=K^{(Ri)!}$ with $R=\lceil 5r\rceil$ and $K$ a large
enough constant, as in the proof of Proposition~\ref{set_unram}, and $P'$ splits in $\F_p[x]$.
\item
Proposition~\ref{pointw} holds with $|\tau_*|\ge (400Q_{i-1}^r)^{-|\mathcal{B}|}$ for $i>2$.
\item
If $x,y\in S_i(J)$ are distinct then $|x-y|\ge (2Q_i)^{-1-1/(\nu_F+1)}$.
\item
$|S_i(J)|\gg Q_{i-1}^{-r} Q_i^{1+1/(\nu_F+1)}/\log Q_i$.
\end{enumerate}

\medskip

These requirements make the difference with respect to the proof in the case $\nu_F=1$ and once the
existence of $S_i(J)$ having these properties is assured, we shall be able to adapt the rest of the
proof.

Firstly, by a weak form of Chebotarev's theorem (see the theorem of Frobenius in \cite{chebotarev}),
the polynomial $P'$ splits  in $\F_p[x]$ for a positive density of the primes. Hence, by the prime
number theorem, the primes $p$ considered in (i) are at least $CQ_i^{1/(\nu_F+1)}/\log Q_i$ with
$C>0$. For any of these primes we have clearly $\nu_{F,p}=\nu_F$ with the notation of
Proposition~\ref{pointw} and its hypotheses are fulfilled. Disregarding a negligible amount of the
primes considered in~(i), the formula in that proposition holds. On the other hand, by
Lemma~\ref{turan} with $I=J$ and $\{\lambda_1,\lambda_2,\dots,\lambda_N\}=\mathcal{B}$, given $q$ we
have $|\tau_*|\ge (|J|/100)^{|\mathcal{B}|}$ for more than $C'Q_i|J|\gg Q_iQ_{i-1}^{-r}$ values of $a$ (note
that $|J|\gg Q_{i-1}^{-r}$ because $J\in \mathcal{G}_{i-1}$).

To sum up, we have got a set of fractions $a/q\in J$ of cardinality comparable to $Q_{i-1}^{-r}
Q_i^{1+1/(\nu_F+1)}/\log Q_i$ such that the formula for $F^*(a/q+h)-F^*(a/q)$ in
Proposition~\ref{pointw} applies with $|\tau_*|\ge (400Q_{i-1}^{-r})^{-|\mathcal{B}|}$ for $i\ge 2$. By
Proposition~\ref{spacing}, we can extract a subset $S_i(J)$ of comparable cardinality, as claimed in
(iv), satisfying the spacing property (iii), in fact we could save an extra logarithm.

\medskip

Now we start our adaptation of the proof of Proposition~\ref{set_unram}.

\smallskip

By the construction of $S_i(J)$, especially note the spacing property (iii), the number $m_i$ of
brothers of a given $I_{a/q}\in \mathcal{G}_i$ satisfies
\begin{equation}\label{brothers_r}
\frac{Q_i^{1+1/(\nu_F+1)}}{CQ_{i-1}^{r}\log Q_i}
<m_i<
C\frac{Q_i^{1+1/(\nu_F+1)}}{Q_{i-1}^{r}\log Q_i}
\qquad\text{for certain $C>0$,}
\end{equation}
and the distance between each couple of brothers is greater than $\delta_i=\frac
14Q_i^{-1-1/(\nu_F+1)}$. On the other hand, by Proposition~\ref{osc1} we have for
\[
 B_q(\epsilon)=\{a\le q: \sup_{q^{-r}/4<h<q^{-r}/2} \frac{|F_*(a/q+h)-F_*(a/q)|}{h^{(\alpha-1)/k}  (h^{1/2k}+h^{1/2r}) h^{-\epsilon}  }>1 \}
\]
the bound $|B_q(\epsilon)|\ll q^{1-2\epsilon r}$.

Consider, as in the case $\nu_F=1$, $\mathcal{G}_1^*=\mathcal{G}_1$, $\mathcal{G}_2^*=\mathcal{G}_2$
and for $i>2$, with
the notation of Proposition~\ref{cheby},
\[
\mathcal{G}_i^*=
\Big\{
I_{a/q}\in \mathcal{G}_{i}\;:\; I_{a/q}\subset I_{\widetilde{a}/\widetilde{q}} \in
\mathcal{G}_{i-1}\text{ with }
a\not\in B_q(4/i)\cup C_{q,\widetilde{q}}(4/i)
\Big\}.
\]
In this case, by Proposition~\ref{cheby} and the bound we just gave for $B_q(\epsilon)$ (due to the rapid growth of $Q_i$), we have for $d_i=|\mathcal{G}_{i}|-|\mathcal{G}_{i}^*|$,
\[
d_i\ll Q_{i-1}^{-8/i}(\log Q_i)
\sum_{q\asymp Q_i}
\sum_{\widetilde{q}\asymp Q_{i-1}}
\big|A\big(\widetilde{q}^r\big)\big|
\]
where $q$ and $\widetilde{q}$ are the denominators as in (i) appearing in the fractions of $S_i(J)$,
$J\in \mathcal{G}_{i-1}$ and in those of $S_{i-1}(J)$, $J\in \mathcal{G}_{i-2}$, respectively. Then,
as we have already seen,  $q$ and $\widetilde{q}$ runs over sets of cardinality
$Q_i^{1/(\nu_F+1)}/\log Q_i$ and $Q_{i-1}^{1/(\nu_F+1)}/\log Q_i$. We have
\[
d_i\ll 
Q_{i-1}^{-8/i-r}
\big(Q_{i-1}Q_i\big)^{1+1/(\nu_F+1)}(\log Q_{i-1})^{-1}.
\]
By \eqref{brothers_r},
\[
\frac{d_i}{m_2m_3\cdots m_i}
\ll (2C)^i
Q_{i-1}^{-8/i}Q_1^{1+1/(\nu_F+1)}\frac{\log Q_i}{\log Q_1}
\prod_{j=1}^{i-2}
\big(Q_j^{r-1-1/(\nu_F+1)}\log Q_j\big).
\]
We are under the conditions of Lemma~\ref{cantor2} with $\delta_n=(2Q_n)^{-1-1/(\nu_F+1)}$ by~(iii),
and then, recalling the notation $r={2s}/(\nu_F+1)$ and  \eqref{brothers_r},
\[
 \dim_HG^*\ge
\liminf_{n\to \infty}
\frac{\log m_{n-1}}{-\log(m_{n}\delta_{n})}
\ge 
\frac{\log\big(Q_{n-1}^{(\nu_F+2)/(\nu_F+1)}Q_{n-2}^{-r}\big)}{-\log Q_{n-1}^{-r}}
=\frac{\nu_F+2}{2s}.
\]
On the other hand, by Lemma~\ref{cantor}, again by \eqref{brothers_r},
\[
 \dim_HG^*\le
 \dim_HG
 \le
\limsup_{n\to \infty}
\frac{\log Q_n^{1+1/(\nu_F+1)}Q_{n-1}^{-r} }{-\log Q_n^{-r}}
=\frac{\nu_F+2}{2s}.
\]

\medskip

If we check $\beta_F(x)=(\alpha-1)/k+1/2s$ for every~$x\in G^*$ then we can take $C_s=G^*$ to finish
the proof. This is done as in the case $\nu_F=1$ (end of Proposition~\ref{set_unram}) but appealing
to Proposition~\ref{pointw} instead of Proposition~\ref{asymp}.
Namely, given $x\in C_s=G^*$, we have  $x\in I_{a_i/q_i}\in\mathcal{G}_i^*$ and taking
$h_i=x-a_i/q_i$, Proposition~\ref{pointw} gives for $i$ large
\[
|F^*(x-h_i)-F^*(x)|=\big|F^*(\frac{a_i}{q_i}+h_i)-F^*(\frac{a_i}{q_i})\big|
\gg \frac{|\tau_*|}{p} h_i^{(\alpha-1)/k}
\]
where $p=q_i^{1/(\nu_F+1)}\asymp h_i^{-1/r(\nu_F+1)}$ and 
$F^*$ is defined modulo $p$. Moreover for $F_*=F-F^*$, since $a_i\not\in B_{q_i}(4/i)$, we have 
\[
 |F_*(x-h_i)-F_*(x)|=\big|F_*(\frac{a_i}{q_i}+h_i)-F_*(\frac{a_i}{q_i})\big|
\ll h_i^{(\alpha-1)/k} (h_i^{1/2k}+h_i^{1/2r}) h_i^{-4/i}.
\] 
According to (i) and (ii), $|\tau_*|>
Ch_i^{\epsilon_i}$ with $\epsilon_i\to 0$ as $i\to\infty$.  
Hence, since we are in the range $r(\nu_F+1)>2k$, $\beta_F(x)\le (\alpha-1)/k+1/r(\nu_F+1) =(\alpha-1)/k+1/2s$.

The proof of the lower bound  $\beta_F(x)\ge (\alpha-1)/k+1/2s$ is identical to that in the case
$\nu_F=1$ (see the last lines in  the proof of Proposition~\ref{set_unram}).
\end{proof}

\begin{proof}[Proof of Theorem~\ref{mainth}]
Write $\beta=1/2r$ with $r>k$. Then
\[
 d_F\big(\frac{\alpha-1}{k}+\beta\big)
=
\dim_H\big\{x\; :\; \beta_F(x)=\frac{\alpha-1}{k}+\frac 1{2r}\big\}
\ge \dim_H C_r
\]
where $C_r$ is the set in Proposition~\ref{set_unram} or Proposition~\ref{set_ram},  that verifies
$\dim_HC_r=(\nu_0+2)/2r$.
\end{proof}

\bigskip

For the proof of Theorem~\ref{mainth_upper} we employ an observation about continued fractions
related to the sets considered by  Jarn\'{\i}k
\[
E_r=\Big(\big\{x\;:\;
\big|c_0x-\frac aq\big|<q^{-r}
\text{ for infinitely many $\dfrac aq\in\Q$}\big\}\cup \Q\Big)\cap [0,1].
\]

\begin{lemma}\label{contfrac}
If $x\not\in E_r\cup\Q$, $r>2$, and $\{a_n/q_n\}_{n=1}^\infty$ are the convergents of~$x$, then for
$n$ large enough
\[
q_n^{-r}\le
\big|x-\frac{a_n}{q_n}\big|
<
\big|x-\frac{a_{n-1}}{q_{n-1}}\big|
<
q_n^{-r'}
\]
where $r$ and $r'$ are H\"older conjugate, i.e., $1/r+1/r'=1$.
\end{lemma}

\begin{proof}
The first inequality follows from the definition of $E_r$ and the second comes from the basic
properties of continued fractions, that also assure $|x-a_{n-1}/q_{n-1}|<1/(2q_nq_{n-1})$. Using
$|x-a_{n-1}/q_{n-1}|\ge q_{n-1}^{-r}$ (because $x\not \in E_r$) we get the third inequality.
\end{proof}

\medskip

\begin{proof}[Proof of Theorem~\ref{mainth_upper}]

According to Proposition~\ref{weyl}, 
\[
 \beta_F(x)\ge \frac{\alpha-1}{k}+\frac{2^{-k}}{r-1}
\qquad\text{for every }x\not\in E_r\text{ and }2<r<1+k/2.
\]
Hence, for any small enough $\epsilon>0$,
\[
  d_F\Big(\frac{\alpha-1}{k}+\frac{2^{-k}}{r-1}\Big)
\le 
\dim_H\Big\{x\;:\; \beta_F(x)<\frac{\alpha-1}{k}+\frac{2^{-k}}{r-\epsilon-1}\Big\}
\le
\dim_H E_{r-\epsilon}.
\]
By Jarn\'{\i}k Theorem \cite[\S10.3]{Fal} $\dim_H E_{r-\epsilon}=2/(r-\epsilon)$ and letting
$\epsilon\to 0$ and writing $\beta=2^{-k}/(r-1)$, we conclude
\[
 d_F\Big(\frac{\alpha-1}{k}+\beta\Big)\le \frac{2\beta}{2^{-k}+\beta}\qquad
\text{for every }0<\beta<\frac{2^{-k}}k.
\]

The case $\beta=0$ follows with similar considerations letting $r\to \infty$ in 
\[
  d_F\Big(\frac{\alpha-1}{k}\Big)
\le 
\dim_H\Big\{x\;:\; \beta_F(x)<\frac{\alpha-1}{k}+\frac{2^{-k}}{r-1}\Big\}
\le
\dim_H E_{r}.
\]

\medskip

On the other hand, for $x\not\in E_r$, $r>2$, and $|h|$ small enough let $a/q$ the
first convergent of
$x$ satisfying $|x-a/q|<|h|/2$. We write
\begin{equation}
\label{triangF}
|F(x+h)-F(x)|\le |F(a/q+h_1)-F(a/q)|+|F(a/q+h_2)-F(a/q)|
\end{equation}
where $h_1=x+h-a/q$ and $h_2=x-a/q$. By Lemma~\ref{contfrac}, 
$2q^{-r}<|h|<2q^{-r'}$. Hence 
$q^{-r}<|h_1|<3q^{-r'}$ and  $q^{-r}<|h_2|<q^{-r'}$.

If the convergents of $x$ from one onward verify
\begin{equation}
\label{supineq}
\sup_{q^{-r}<|h|<3q^{-r'}}
\big|F(\frac aq+h)-F(\frac aq)\big| |h|^{-(\alpha-1/2)/k}\le q^{1-2/r},
\end{equation}
then substituting in (\ref{triangF})
\[
|F(x+h)-F(x)|\le
2q^{1-2/r} |h|^{(\alpha-1/2)/k}
\le
 |h|^{(\alpha-1/2)/k-(1-2/r)/r'}.
\]
Let $A_n$ be the set of the irreducible fractions $a/q\in [0,1]$, $2^n\le q<2^{n+1}$ not
satisfying
(\ref{supineq}), and let $A$ be the set of $x\in [0,1]$ having convergents in $A_n$ for infinitely
many values of $n$. Then
\begin{equation}
\label{inclEA}
\big\{
x\;:\; \beta_F(x)<	\frac{\alpha-1/2}{k}-\frac{1}{r'}\big(1-\frac 2r\big)
\big\}
\subset E_r\cup A.
\end{equation}

Clearly, for any $m$, $A$ admit the covering
\[
 A\subset \bigcup_{n=m}^\infty \bigcup_{a/q\in A_n}
\big(
\frac aq-\frac{1}{q^2}
,\frac aq+\frac{1}{q^2}
\big).
\]
By Proposition~\ref{osc2}, the cardinality of $A_n$ is $O\big(2^{n(4/r+\epsilon)}\big)$ as the
length of each intervals with $a/q\in A_n$  is $O\big(2^{-2n}\big)$ we have, letting $m\to\infty$,
that the $(2/r+\epsilon)$-Hausdorff measure of $A$ is zero for every $\epsilon>0$. Hence 
\[
d_F\Big(
\frac{\alpha-1/2}{k}-\frac 1{r'}\big(1-\frac 2r\big)
\Big)\le \dim_{\text{H}}(E_r\cup A)=\frac 2r.
\]
This gives
\[
d_F\big(\beta+\frac{\alpha-1}{k}\big)\le 
\frac 32-\sqrt{\frac{k+4}{4k}-2\beta} 
\qquad\text{for }\quad 0<\beta< \frac{1}{2k}
\]
writing $\beta=1/2k-(1-r^{-1})(1-2r^{-1})$.
\end{proof}

\section{Heuristics}

Our aim in this section is to comment our expectations regarding the true nature of $d_F$ and to explain our way of proceeding in the paper in light of them. 

\

A direct application of Parseval formula proves
\[
 \Big(\int_0^1\big|F(x+h)-F(x)|\; dx\Big)^{1/2}\sim
Ch^{\rho +1/2k}
\qquad\text{where } \rho=\frac{\alpha-1}{k},
\]
then we expect typically the H\"older exponent to be $\beta_F(x)=\rho+1/2k$. This is consistent with $\omega(\frac{1}{2k}^-)=1$ in Theorem~\ref{mainth_upper}. In fact, as we shall see later, it is likely that $\beta_F(x)\le \rho+1/2k$ for every irrational value $x$. On the other hand, it can be proven as in \cite[Corollary~2.3]{chaubi} that the points $a/q$ with $\tau_0=0$ have exponent $(\rho+1/2k)k/(k-1)$, the rest having exponent $\rho$.

As we mentioned in the introduction, the integration based techniques are wasteful here because they overlook $0$-measure sets and hence are incapable of detecting the fractal sets determined by H\"older exponents different from $\rho+1/2k$ that we need to prove the multifractal nature of $F$.

Given an irrational value $x$ and  $h>0$ small, we think that one can understand $F(x+h)-F(x)$ by looking at
$F(a/q+h)-F(a/q)$, where $a/q$ is the convergent in the continued fraction of $x$ nearest to $x+h$. 
This implies
\begin{equation}\label{range_continued}
 q^{-r} \ll h \ll (q q')^{-1}           \qquad r\ge 2,
\end{equation}
where $q^{-r}=|x-a/q|$ and $a'/q'$ is the previous convergent of $x$.

The factor $e\big(P(n)h\big)-1$ appearing in the series expansion of $F(a/q+h)-F(a/q)$  is small when $n$ is much smaller than $h^{-1/k}$ and the coefficients of the series decay when $n$ is large. This suggest that
we can model $F(a/q+h)-F(a/q)$ by $h^{\rho} S$ with
\[
 S= h^{1/k}\sum_{n\asymp h^{-1/k}} e(\frac{a P(n)}{q}).
\]

The range (\ref{range_continued}) includes $h\gg q^{-2}$. There $S$ is a short sum, and if we assume  nothing better than square root cancellation, i.e., $S\gg h^{1/2k}$, then we would get $\beta_F(x)\le \rho+1/2k$ as claimed above. 
We are not able to prove this lower bound for $S$ because  we do not have a \lq\lq trivial main term'' that we can separate.  On the other hand, it should be possible to push the methods in this paper to prove that $d_F(\rho+1/2k)\ge \nu_0/2k$.

We assume $r>k$, then when $h$ is close to its lower limit in \eqref{range_continued}, we have that the length of the range of summation is greater than $q$ and the periodicity leads to think that $S$ can be compared with the complete sum~$\tau_0$
\[
S\asymp q^{-1} \sum_{n \pmod{q}} e(\frac{a P(n)}q) =  q^{-1} \tau_0.
\]
Now, for prime $q$ there is always at least and typically at most, as shown in Lemma~\ref{lbtau}, square root cancellation in $\tau_0$, so we have $S\asymp  q^{-1/2}$ which 
gives $S\asymp h^{1/2r}$ under $h\asymp q^{-r}$. With this information we can build (see Proposition~\ref{set_unram}) a set with Hausdorff dimension $2/r$ whose elements $x$  have H\"older exponent $\beta_F(x)=\rho + 1/2r$, hence the lower bound $4\beta$ obtained in the main theorem for $d_F(\rho+\beta)$ follows.

The square root cancellation philosophy fails drastically for prime powers under  conditions depending on the fine structure of the polynomial $P$. In particular, when $\nu_F>1$ we
expect the main contribution in $S$ to come from arithmetic progressions in $n$ for which the phase of the exponential is essentially constant modulo 1. The simplest example is $P(x)=x^k$. In this case, by choosing $q=p^j$ with $p$ prime and $j\le k$ it follows that the exponential restricted to the sequence $n\equiv 0 \pmod{p}$ is constant, that is
\[
 h^{1/k}\sum_{\substack{n\asymp h^{-1/k}\\ n\equiv 0 \pmod{p}}} e(\frac{a n^k}{p^j}) = h^{1/k}\sum_{l\asymp h^{-1/k}/p} e(a p^{k-j} l^k) \asymp p^{-1},
\]
so in that part of the sum there is no cancellation. We think that the rest of $S$ gives a smaller contribution (actually in Proposition~\ref{osc1}
we proved it in average over $a$) and then we should get
$ S\asymp p^{-1} = q^{-1/j}$
that for $h\asymp q^{-r}$ reads $S\asymp  h^{1/jr}$. Proceeding as before, we can construct
a set of $x$ of dimension $(1+1/j)/r$ and constant H\"older exponent $\beta_F(x)=\rho+1/jr$, and then the lower bound  $(j+1)\beta$ for $d_F(\rho+\beta)$ follows. 
The best choice of $j$ is, of course, the largest value, that in our range is $j=k$ and corresponds to $j=\nu_F+1$ in the setting of the main theorem, as fixed in \S3. In principle one may think that $j>k$ could give a stronger lower bound but in this range the phase is not going to stay constant so there will be further cancellation; moreover, 
by the scarcity of the rationals of the form $a/p^j$, the
dimension is going to be smaller. 
By the chinese remainder theorem, the number of terms in special arithmetical progressions behaves multiplicatively, so taking $q$ as a prime power is the best choice to prevent cancellation in $S$ and even to get stronger lower bounds for $d_F$.

For some time we thought that the exponential can be essentially constant in a large arithmetic progression only if $P'$ has zeros of high order, this and the periodicity give the lower bound for $d_F$ proven in the paper. That guess was based on the fact that we had proven it \emph{in average over $a$} (see Proposition \ref{cheby}). Afterward we realized that it can also happen if some of the higher derivatives of $P$ has zeros of large order; but in that case it will happen for fractions $a/p^j$ with \emph{just very special} $a$ determined by arithmetic conditions. This will not give a better lower bound for $P(x)=x^k$, but it will do for instance for $P(x)=x^k+x^2$ with large $k$. In this last case the zeros of $P'$ are all of order one, so the lower bound for $d_F$ from the paper would be $4\beta$. But, although $P'$ does not have high order zeros, since $P'''$ does at $x=0$, the monomial $n^k$ in the phase can be ``deleted'' for $n\equiv 0 \pmod{p}$ by taking $q=p^k$:
\[
h^{1/k} \sum_{n\asymp h^{-1/k}, \, n\equiv 0 \pmod{p}} e(\frac{a (n^k+n^2)}{p^k}) = h^{1/k} \sum_{l\asymp h^{-1/k}/p} e(\frac{a p^2}{p^k}l^2).
\]
As we know, for most $a$ there is further cancellation in this sum, but not for the $a$ verifying $\langle ap^2/p^k \rangle \ll (h^{-1/k}/p)^{-2}$. There are around $q(h^{-1/k}/p)^{-2}$ of them. Actually, we can parametrize them as
\begin{equation}\label{aspecial}
 \frac{a}{p^k}=\frac{c}{p^2} + \frac{s}{p^k} \qquad 1\le c\le p^2 \qquad 1\le s \ll q (h^{-1/k})^{-2},
\end{equation}
with $s$ coprime to $p$. For these values, $S\asymp p^{-1}\asymp h^{1/kr}$ under $h\asymp q^{-r}$and, as before, we could build a set of $x$ with constant H\"older exponent $\beta_F(x)=1/kr$ and dimension
$
 ( 1-2(r-1)/k +1/k)/r,
$
since this time the number of the possible values of $a$ is $q^{1-2(r-1)/k}$. This would give a lower bound of $(k+3)\beta-2/k$ for $d_F(\rho+\beta)$, which is larger than our previous bound $4\beta$ in the range $\beta>2/k(k-1)$. 

This last example suggests that, for some polynomials, the lower bound given in Theorem~\ref{mainth} is not going to coincide with the true value of $d_F$. 
The obstruction to prove with our approach that it actually happens for $P(x)=x^k+x^2$ is that the rational numbers defined by \eqref{aspecial} are quite sparse and it seems difficult to get average results over them. For some very special polynomials, like $P(x)=x^k+x$, one can really prove that the lower bound of Theorem~\ref{mainth}  is not the real value of $d_F$ in some range by using Poisson summation and a precise knowledge of the coefficients $\tau_m$. 
In particular, one cannot expect an exact formula for the spectrum of singularities of $F$ only depending on $\nu_F$.

Even though we know that the lower bound from the paper is in general not sharp on the whole range, we think that it is so when $\beta$ is near zero. Our reasoning is that $S$ should always be bounded by $q^{-1/k}$ (notice that this equals $p^{-1}$ for $q=p^k$) which is $h^{1/rk}$ for $h\asymp q^{-r}$ and then gives an exponent of at least $\rho+\beta$ with $\beta=1/kr$. But then $\beta$ can be small only if $r$ is large, 
this means that we have very good rational approximations and Poisson summation  (Proposition~\ref{poisson}) allows to see that $S$ depends just on $\tau_0$, and the size of $\tau_0$ is controlled by the largest order of a zero of $P'$ (see Theorem~2 in \cite{loxvau}).

For $k>2$, we also think that $d_F$ should always have a discontinuity at $\rho+1/2k$, with $d_F(\rho+1/2k)=1$ and $d_F(\rho+1/2k^{-})<1$. This would represent a contrast with the quadratic case. The idea is that deviations from square root cancellation should come either from $r>k$ or from special $q$ (like prime powers), and both cases give small Hausdorff dimensions.

The behavior of $d_F$ for $P(x)=(x^2+1)^d+ x^2$ should be similar to $P(x)=x^k+x^2$; in both cases, we expect the graph of $d_F(\rho+\beta)$ to consist of two segments in the range $0<\beta<1/2k$, the lower bound from Theorem~\ref{mainth} for $\beta$ small and the one from the special $a$'s otherwise. But there are even more complex examples, like $P'''(x)=(x^2+1)^d+x^2$; for this case it is even difficult to guess, when taking the rationals $a/p^d$, the density of $a$ for which the phase remains essentially constant for the arithmetic progression $n\equiv b \pmod{p}$, with $b$ a root of $x^2+1$ modulo $p$, due to the fact that we do not control the location of $b$ in the interval $[1,p]$.

The previous examples suggest that the spectrum of singularities of $F$ depends on fine points about the structure of the polynomial $P$. On the other hand, these examples are very artificial and involve terms with unbalanced multiplicities. By this reason, we think that the lower bound in Theorem~\ref{mainth} becomes an equality in the whole range for most polynomials.

\bibliography{./spectrum}{}
\bibliographystyle{alpha}

\end{document}